\documentclass[11pt]{amsart}
\usepackage{latexsym, amsmath, amssymb,amsthm,amsopn,amsfonts, eucal}
\usepackage{epsfig,graphics,color,graphicx,graphpap}
\usepackage{amssymb}
\usepackage{todonotes}

 \definecolor{skyblue}{rgb}{0.85,0.85,1}

\ifx\pdfoutput\undefined
  \DeclareGraphicsExtensions{.eps}
\else
  \ifx\pdfoutput\relax
    \DeclareGraphicsExtensions{.eps}
  \else
    \ifnum\pdfoutput>0
      \DeclareGraphicsExtensions{.pdf}
    \else
      \DeclareGraphicsExtensions{.eps}
    \fi
  \fi
\fi

\newtheorem{theorem}{Theorem}[section]

\newtheorem{definition}[theorem]{Definition}

\newtheorem{proposition}[theorem]{Proposition}

\newtheorem{corollary}[theorem]{Corollary}

\newcommand{\RP}{{\mathbb{RP}}}

\newcommand{\SP}{{\mathbb S}}

\newcommand{\tr}{\operatorname{tr}}

\newcommand{\abs}[1]{\left\vert{#1}\right\vert}

\newcommand{\Mcross}{\mathbb{M}^n_{\mathrm{cross}}}
\newcommand{\Mrelax}{\mathbb{M}^n_{\mathrm{relax}}}
\newcommand{\Mtworelax}{\mathbb{M}^2_{\mathrm{relax}}}
\newcommand{\Mcrossthree}{\mathbb{M}^3_{\mathrm{frame}}}
\newcommand{\Mthreerelax}{\mathbb{M}^3_{\mathrm{relax}}}

\newcommand{\Mbold}{\mathbb{M}}
\newcommand{\Mall}{\mathbb{M}^n_{\mathrm{all}}}
\newcommand{\Mallnsqr}{\mathbb{M}^{n^2}_{\mathrm{all}}}

\newcommand{\Mtr}{\mathbb{M}_{\mathrm{tr}}}
\newcommand{\bX}{\mathbf{X}}

\newcommand{\LC}{\left(}
\newcommand{\LV}{\left|} 
\newcommand{\RC}{\right)}
\newcommand{\RV}{\right|}

\newcommand{\mba}{\mathbf{a}}
\newcommand{\mq}{\mathbf{q}}

\newcommand{\R}{{\mathbb R}}

\newcommand{\N}{{\mathbb N}}

\newcommand{\e}{\varepsilon}

\theoremstyle{plain}

\newtheorem{prop}{Proposition}[section]

\newtheorem{lem}[prop]{Lemma}

\theoremstyle{definition}

\newtheorem{rem}{Remark}[section]

\numberwithin{equation}{section}

\def\squarebox#1{\hbox to #1{\hfill\vbox to #1{\vfill}}}

\newcommand{\p}{\partial}
\usepackage{amsxtra}

\thanks{We wish to thank Braxton Osting and Ryan Viertel for introducing this problem to the authors. { We also wish to thank David Palmer and Justin Solomon for helpful discussions related to their work, along with the comments from the anonymous referees.}  The first author was supported in part by NSF grant DMS-1615952. The third author was supported in part by NSF grants DMS-1516565 and DMS-2009352. The authors would like to thank the IMA where the project was initiated.}

\begin{document}    

\title{A variational method for generating $n$-cross fields
using higher-order $Q$-tensors}

\author{Dmitry Golovaty}
\address{Department of Mathematics, The University of Akron, Akron, OH 44325, USA}
\email{dmitry@uakron.edu}

\author{Jose Alberto Montero}
\address{Facultad de Matem\'aticas, Pontificia Universidad Cat\'olica de Chile,
Vicu\~na Mackenna 4860, San Joaqu\'in, Santiago, Chile.}
\email{amontero@mat.puc.cl}

\author{Daniel Spirn}
\address{School of Mathematics, University of Minnesota, Minneapolis, MN 55455}
\email{spirn@umn.edu}

\begin{abstract}
{An $n$-cross field is a locally-defined orthogonal coordinate system invariant with respect to the cubic symmetry group.  Cross fields are finding wide-spread use in mesh generation, computer graphics, and materials science among many applications. It was recently shown in \cite{Schaft} that $3$-cross fields can be embedded into the set of symmetric $4$th-order tensors. The concurrent work \cite{palmer2019algebraic} further develops a relaxation of this tensor field via a certain set of varieties.  In this paper, we consider the problem of generating an arbitrary $n$-cross field using a fourth-order $Q$-tensor theory that is constructed out of tensored projection matrices. We establish that by a Ginzburg-Landau relaxation towards a global projection, one can reliably generate an $n$-cross field on arbitrary Lipschitz domains. Our work   provides  a rigorous approach that offers several new results including porting the tensor framework to arbitrary dimensions, providing a new relaxation method that embeds the problem into a global steepest descent, and offering a relaxation scheme for aligning the cross field with the boundary. Our approach is designed to fit within the classical Ginzburg-Landau PDE theory, offering a concrete road map for the future careful study of singularities of energy minimizers.}
\end{abstract}

\maketitle

\section{Introduction}
{In this paper we use variational methods for tensor-valued functions in order to construct $n$-cross fields in $\R^n$.  Loosely speaking, an $n$-cross field associates a set of $n$ orthogonal lines with every point in $\R^n$.} A particular question of interest is whether it is possible to construct a smooth field of $n$-crosses in $\Omega$, assuming certain behavior of that field on $\partial\Omega$. This problem has received a considerable attention in computer graphics and mesh generation---see for example the review of the many applications of cross and frame fields in \cite{Vaxman}. In two dimensions (or on surfaces in three dimensions) { quad} meshes can be obtained by finding proper parametrization based on a $2$-cross field defined over a triangulated surface \cite{Li:2012:AMU:2366145.2366196}.  

A similar two step procedure in three dimensions has been proposed recently by a number of authors with the aim to generate hexahedral meshes. First, a 3-frame field is constructed by assigning a frame to each cell of a tetrahedral mesh, then a parametrization algorithm is applied to generate a hexahedral mesh \cite{05_IMR23_Kowalski,Nieser_2011}. From a mathematical point of view, the first step in this procedure requires one to construct a 3-cross field in $\Omega\subset\R^3$ that is sufficiently smooth and properly fits to $\partial\Omega$, e.g., by requiring that one of the lines of the field is orthogonal to $\partial\Omega$. Generally, a cross field that satisfies this type of the boundary condition has singularities on $\partial\Omega$ and/or in $\Omega$ due to topological constraints, as follows from an appropriate analog of the Hairy Ball Theorem (see Section~\ref{sec:boundary}). 

A number of approaches have been proposed to construct a $2$- or $3$-frame and cross fields or their analogs. Some schemes involve identification of the field on the boundary and its subsequent reconstruction in the interior of the domain, for example, by using 
an optimization procedure \cite{bommes}. In three dimensions, the first task can be accomplished by looking for the harmonic map on the boundary surface that has one of the $3$-frame vectors orthogonal to the boundary \cite{BERNARD2014175}, by  prescribing the $3$-frame field on the boundary \cite{Li:2012:AMU:2366145.2366196}. The reconstruction of the $3$-frame field in the interior is achieved by propagating the frame from the boundary and then optimizing its smoothness by minimizing a function that, e.g., penalizes for frame changes in the neighboring tetrahedra \cite{Li:2012:AMU:2366145.2366196,BERNARD2014175}. Other, $2$-frame reconstruction  algorithms over surfaces rely on solving a Ginzburg-Landau equation \cite{ViertelOsting,BEAUFORT2017219}. Some authors do not distinguish between the frames in the interior of the domain and on its boundary and simply optimize the frame distribution via energy minimization \cite{05_IMR23_Kowalski}; here the $3$-frame has also been described by using spherical harmonics \cite{Huang:2011:BAS:2070781.2024177}. 

The related recent work has been done on singularity-constrained octahedral fields for hexahedral meshing \cite{Liu:2018:SOF:3197517.3201344}, boundary element octahedral fields in volumes \cite{Solomon:2017:BEO:3087678.3065254}, smoothness driven frame field generation for hexahedral meshing \cite{KOWALSKI201665}, robust hex-dominant mesh generation using field-guided polyhedral agglomeration \cite{Gao:2017:RHM:3072959.3073676}, symmetric moving frames \cite{Corman:2019:SMF:3306346.3323029}, and all-hex meshing using closed-form induced polycube \cite{Fang:2016:AMU:2897824.2925957}.

{As an aside, note that the problem of constructing a 3-cross field subject to prescribed boundary conditions is related to the problem of modeling of dislocation structures in crystalline materials \cite{Bilby1955,Epstein}. We do not pursue this relationship further in the present paper.}

What then is an "optimal" way to automatically generate a 3-cross field that satisfies prescribed boundary conditions and is not too singular? A promising direction was identified in \cite{BEAUFORT2017219,ViertelOsting} { for 2-cross fields} where a connection to the Ginzburg-Landau theory was noticed. This connection is transparent in two dimensions where a frame- or a cross-field is fully defined by a single angle. The appropriate descriptors in three dimensions, however, {was not known until very recently \cite{Schaft,palmer2019algebraic}}.   {The corresponding rigorously justified Ginzburg-Landau relaxation for 3-cross fields remained unclear, and one of our contributions is to present a natural approach based on our prior experience with the Ginzburg-Landau-type theories for vector- and matrix-valued maps.

The primary goal of this work is to propose a unified tensor-based approach to constructing $n$-cross fields that takes advantage of classical PDE theory. To this end, in what follows we will not be interested in implementing the most efficient method of solving relevant boundary value problems that arise within our framework. Instead, we use an off-the-shelf finite element analysis solver (COMSOL \cite{comsol}) to compute the gradient descent in order to arrive at local minimizers of our energy functional and to visualize singular sets that characterize these minimizers.  This numerical strategy could likely be sped up by using, for example, an MBO-type scheme. 
It is unknown whether such MBO-type schemes converge, as $\varepsilon \to 0$ and $t\to \infty$, to the same minimizers as this FEA solver for the Ginzburg-Landau system, and this merits further consideration (see Section~\ref{sec:discussion}). Although an analysis of this system of PDEs is beyond the scope of the present paper, the model that we propose can be analyzed by extending the scope of the existing Ginzburg-Landau theory as we will discuss in a follow-up publication. In particular it is unclear if the minimizers we obtain are global or local, i.e., whether or not the singular structures that we see numerically are ``optimal".

Our work is closely intertwined with that in \cite{Schaft,palmer2019algebraic}. We postpone the discussion of the exact relationship between our methodology and the developments in \cite{Schaft,palmer2019algebraic} until Section \ref{sec:ZSP}, after the necessary ideas and terminology have been introduced.}

\subsection{Our work}

{We begin with the following definitions. Given $k,n\in\N$, a $k$-{\em frame} $F^k$ in $\R^n$ is an ordered set of $k$ vectors in $\left\{{\bf a}_i\right\}_{i=1}^k\subset\R^n$. If the vectors $\left\{{\bf a}_i\right\}_{i=1}^k$ are mutually orthonormal, then we say that the $k$-frame is orthonormal. Associated with each orthonormal frame $F^k$, we define a $k$-{\em cross} $C^k$ as an unordered set of $k$ equivalence classes corresponding to $\left\{{\bf a}_i\right\}_{i=1}^k$ in $\R^n\backslash\{0\}$ under the equivalence relation ${\bf y}\sim\lambda{\bf x}$ for all $\lambda\neq0$. In other words, a $k$-cross is an unordered set of $k$ orthogonal elements of $\mathbb{RP}^{n-1}$; it can also be thought of as $k$ mutually orthogonal lines $\left\{l_i\right\}_{i=1}^k$ in $\R^n$, where $l_i$ is parallel to ${\bf a}_i$ for each $i=1,\ldots,k$. Without loss of generality, in what follows we will always assume that $k=n$ and simply refer to $n$-frames and $n$-crosses. We will also drop superscript $k$ in the respective notation for frames and crosses. Note that an orthonormal $n$-frame is also an orthonormal basis of $\R^n$. Note that the notions of a frame and a cross vary throughout the literature and sometimes these are even used interchangeably. Here we make these notions precise and distinct. With these definitions in hand, we can consider fields of $n$-frames and $n$-crosses on a domain $\Omega\subset\R^n$.}

{We now describe our approach to representing $n$-cross field and their relaxations, which is motivated by our experience with  nematic liquid crystals.} In  nematic liquid crystals, partial orientational order exists within certain temperature ranges so that a nematic sample has a preferred molecular orientation at any given point of the domain it occupies. One possible description of a nematic then utilizes a unit vector field ${\bf n}:\R^3\to\SP^2$ at every point of the domain $\Omega\subset\R^3$; note that this essentially generates a $1$-frame field in $\Omega$. 

The physics of the problem, however, dictates the same probability of finding a head or a tail of a nematic molecule pointing in a given direction, hence the appropriate descriptor of the nematic state must be invariant with respect to inversion ${\bf n}\to-{\bf n}$. The oriented object satisfying this symmetry condition is not ${\bf n}$ but rather the projection matrix ${\bf n}\otimes{\bf n}$ that can also be identified with an element of the projective space $\RP^2$ or a $1$-cross. Thus we can interpret a field of projection matrices on $\Omega$ as a $1$-cross field. We will generalize this connection between the projection matrices and $1$-cross fields to higher dimensional matrices and $n$-cross fields in the remainder of this paper. 

 The connection, in fact, goes a bit deeper if one is interested in exploring singularities of cross fields. As was already alluded to above, a nematic configuration in $\Omega$ satisfying certain boundary conditions on $\partial\Omega$ is generally subject to topological constraints that lead to formation of singularities in $\Omega$. Within a variational theory for nematic liquid crystals one typically assumes that an equilibrium configuration minimizes some form of elastic energy associated with spatial changes of the preferred orientation. In the simplest approximation, this energy reduces to a Dirichlet integral
 \[\int_\Omega{\left|\nabla u\right|}^2 dx \]
 of $u={\bf n}$ or $u={\bf n}\otimes{\bf n}$, depending on the kind of order parameter that one needs. It turns out, however, that for certain types of singularities (e.g., vortices in $\R^2$ or disclinations in $\R^3$) that are topologically necessary, this energy is infinite. One way around this difficulty is to replace the nonlinear constraints on the order parameter field by adding an appropriate, heavily-penalized potential to the energy that forces the constraint to be almost satisfied a.e. in $\Omega$ in an appropriate limit. For example, instead of using a field of projection matrices, the relaxed competitors can be assumed to take values in the space of symmetric matrices $Q$ of trace $1$ satisfying the same linear constraints as the projection matrices. Then the property $P^2-P=0$ of the projection matrices can be enforced by adding the term $\frac{1}{\varepsilon^2}{\left|Q^2-Q\right|}^2$ to the energy
 and letting $\e \to 0$ (cf. \cite{GolovatyMontero}).  This results in a prototypical expression
 \[\mathcal{E}(Q)=\frac{1}{2}\int_\Omega{{\left|\nabla Q\right|}^2+\frac{1}{\varepsilon^2}{\left|Q^2-Q\right|}^2} dx,\]
 that lies at the core of the Ginzburg-Landau-type theory for nematic liquid crystals (with a minor caveat that, for physical reasons, this theory named after Landau and de Gennes, considers translated and dilated version of $Q$ \cite{Newton}). In this paper, we show that exactly the same approach can be undertaken to construct $n$-cross fields in $\R^n$.
 
 { 
Our framework, described in detail in Section~\ref{sec:ncross},  provides a new and rigorous approach that offers some advantages over previous work (though the MBO framework developed in \cite{palmer2019algebraic} may be better optimized for running numerical experiments).  First, our framework applies in arbitrary dimensions, and it naturally encodes the associated 4-tensor that arises due to  boundary conditions. Indeed, the proposed relaxation provides a global energy that can be studied analytically.  Second, the new Ginzburg-Landau relaxation, that embeds the problem into a global steepest descent, allows for a new selection principle for the limiting $n$-cross field.  Finally, we provide additional relaxations schemes that can allow for weak anchoring of  the boundary with the $n$-cross field.}

In Section~\ref{sec:boundary}, we introduce the notion of an $n$-cross on an $n-1$-dimensional manifold and then use this notion in Section~\ref{sec:GLn} in order to define natural boundary conditions for $n$-cross-valued maps. This allows us to formulate a Ginzburg-Landau-type variational problem for relaxed, tensor-valued maps. Sections~\ref{sec:2D} and \ref{sec:3D} are devoted to $2$- and $3$-cross fields, respectively, and Section~\ref{sec:examples} presents several computational examples of $3$-cross field reconstructions in three-dimensional domains. Here we give one example of a tensor-valued solution of the Ginzburg-Landau problem that replicates a setup discussed in \cite{viertel2016analysis} (cf. \cite{raysokolov11}). In Figs.~\ref{fig0}-\ref{fig01} we show the $3$-cross field distribution in the domain consisting of the cube with a notch in the shape of a cylinder. On the boundary, one of the lines of the $3$-cross field is assumed to be perpendicular to the surface of the boundary and the cross field is obtained by solving the system of Ginzburg-Landau PDEs subject to the natural boundary condition. The result reproduces that in \cite{viertel2016analysis}, where it was computed using a different technique. In Fig. \ref{fig:bone} we show the singular set of the solution to the same problem in a bone-shaped domain. The singularity is identical to that found in \cite{palmer2019algebraic} while a different, twisted structure was obtained in \cite{raysokolov11}. 

{Finally, in the Appendix we prove Theorem~\ref{thm:limitrelax}, describe the connection between our approach and the odeco framework found in \cite{palmer2019algebraic}, and present the system of partial differential equations governing gradient descent.}
\begin{figure}
\begin{center}
\includegraphics[scale=.4]{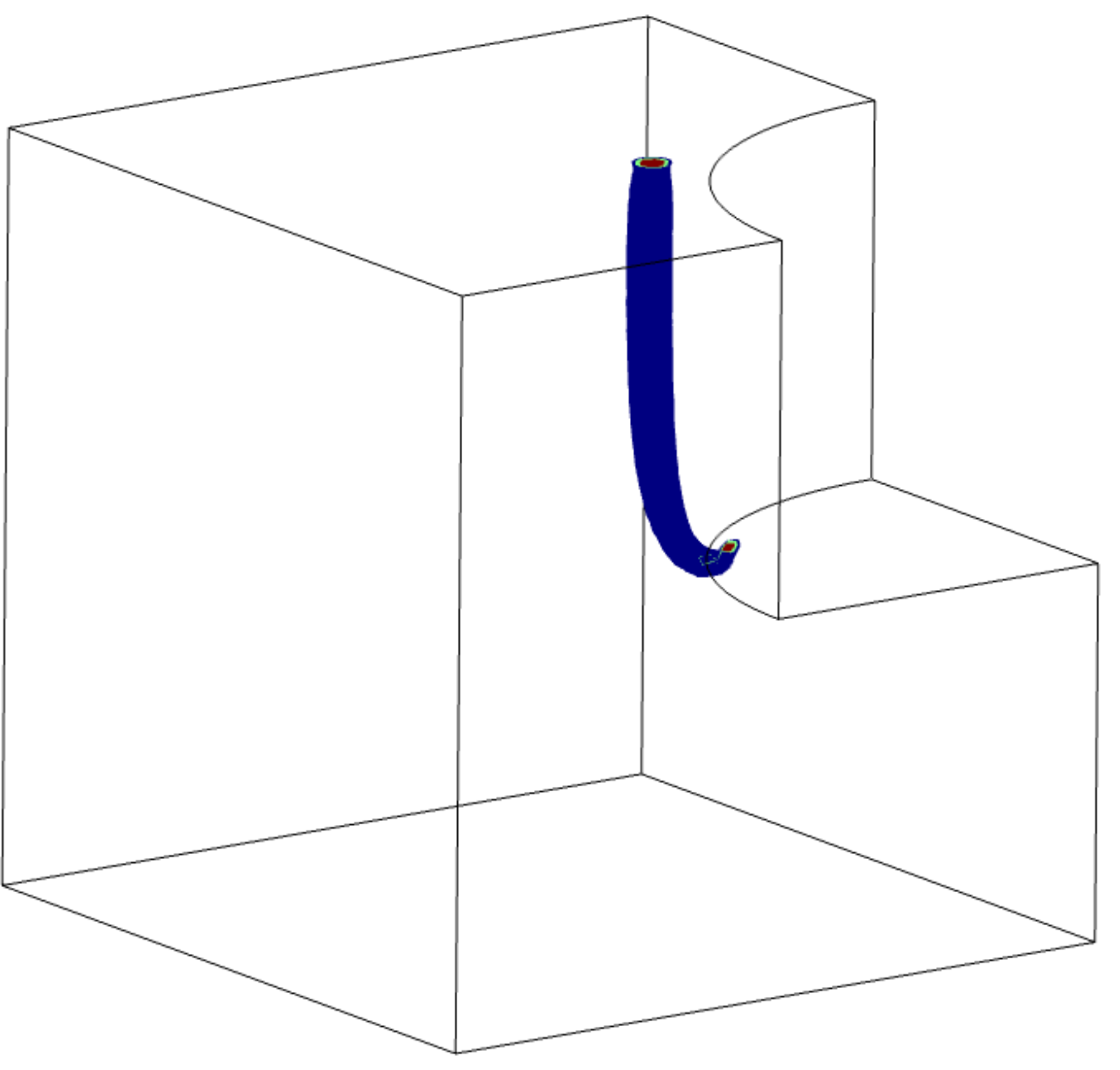} \qquad \includegraphics[scale=.4]{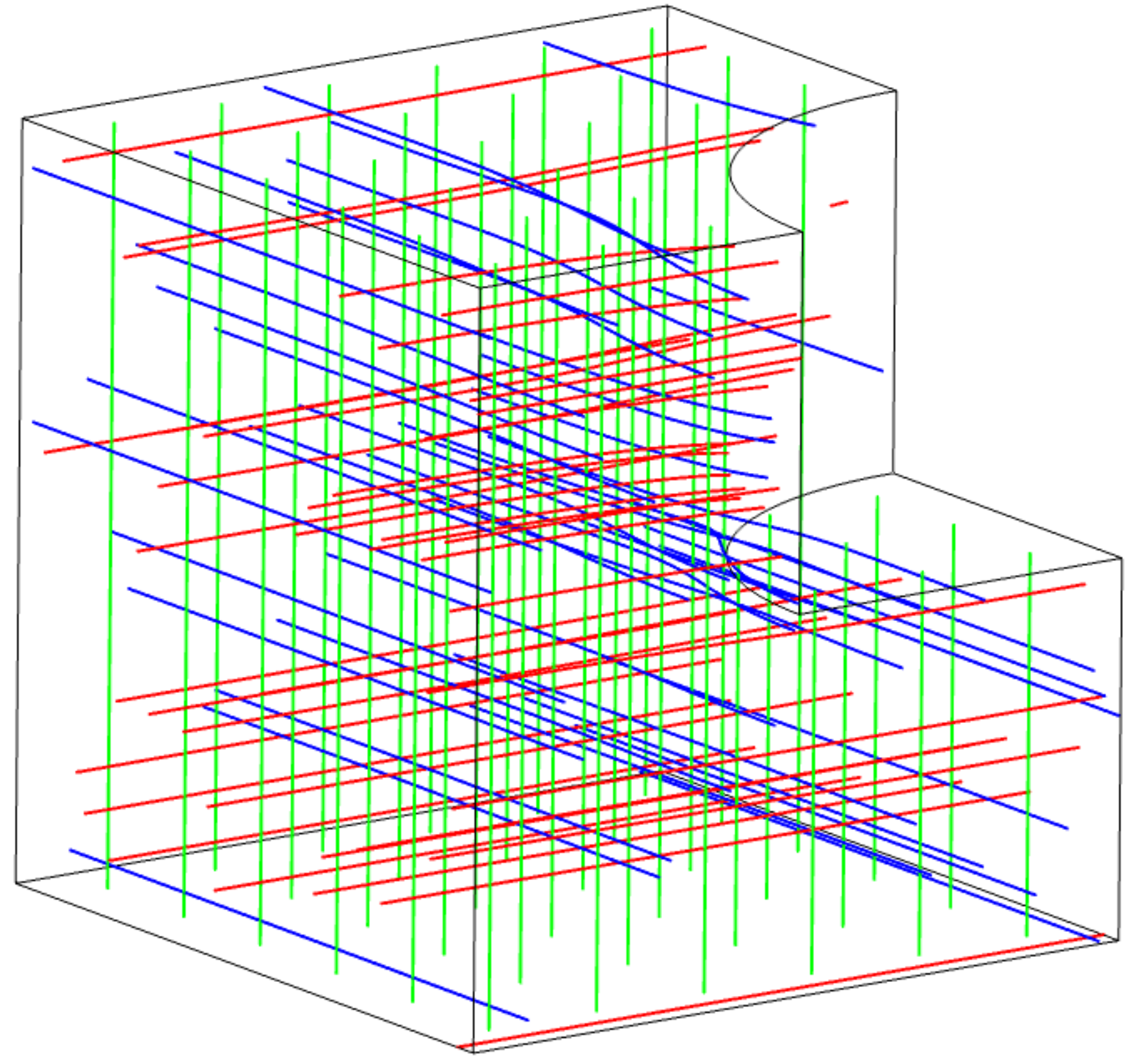}
\caption{Solution of the Ginzburg-Landau PDE subject to the natural boundary conditions in a cube-shaped domain with a cylindrical notch: A disclination in the $3$-cross field via level surface plot for $\frac{1}{\e^2}|\mathcal{Q}^2 - \mathcal{Q}|^2$ (left); Streamlines of the $3$-cross solution field with colors representing three different directions followed along the cross field (right).}\label{fig0}
\end{center}
\end{figure}
\begin{figure}
\begin{center}
\includegraphics[scale=.4]{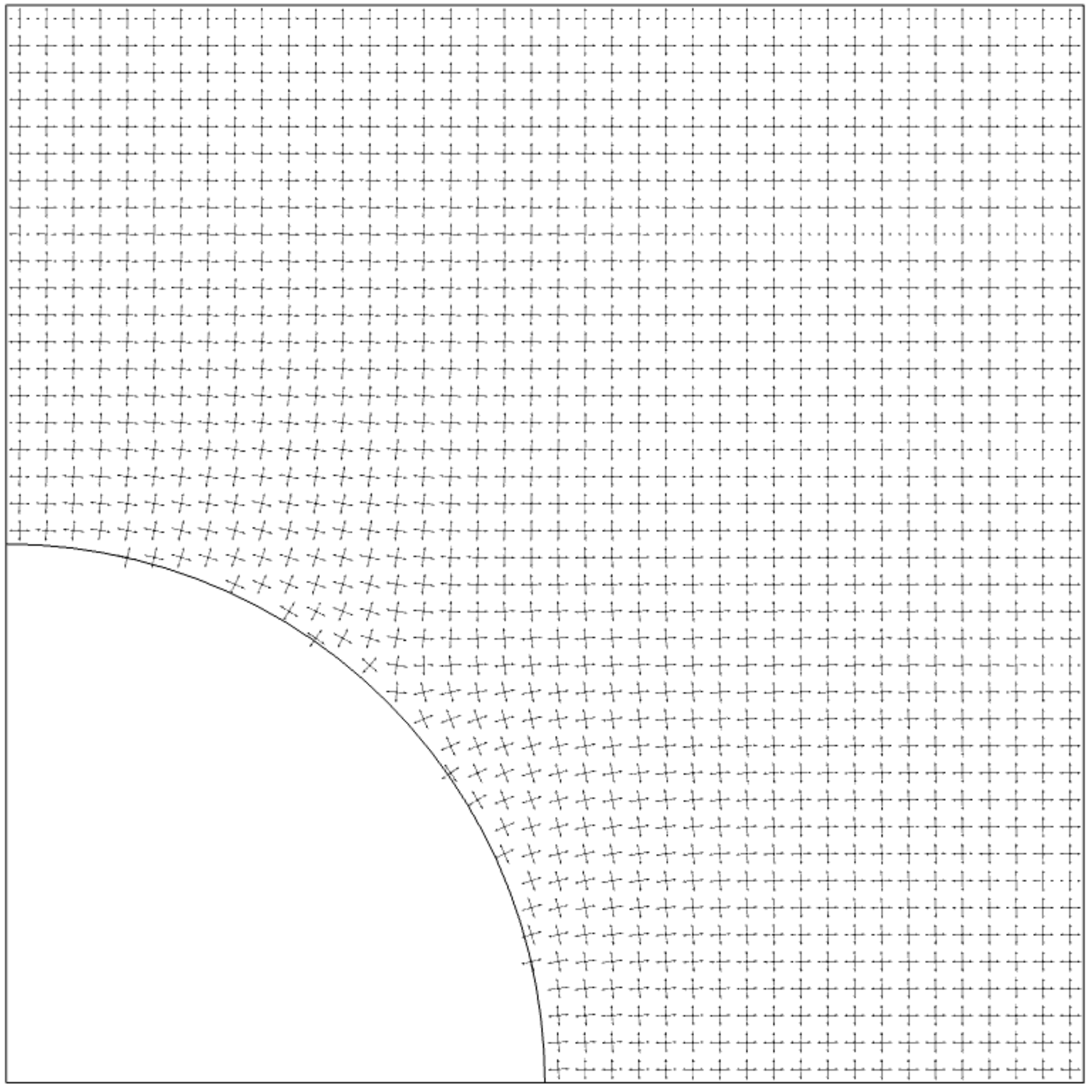}
\caption{Solution of the Ginzburg-Landau PDE subject to the natural boundary conditions in a cube-shaped domain with a cylindrical notch: A horizontal cross-section of the $3$-cross solution field at the level intersecting the notch.}\label{fig01}
\end{center}
\end{figure}
\begin{figure}
\begin{center}
\includegraphics[scale=.4]{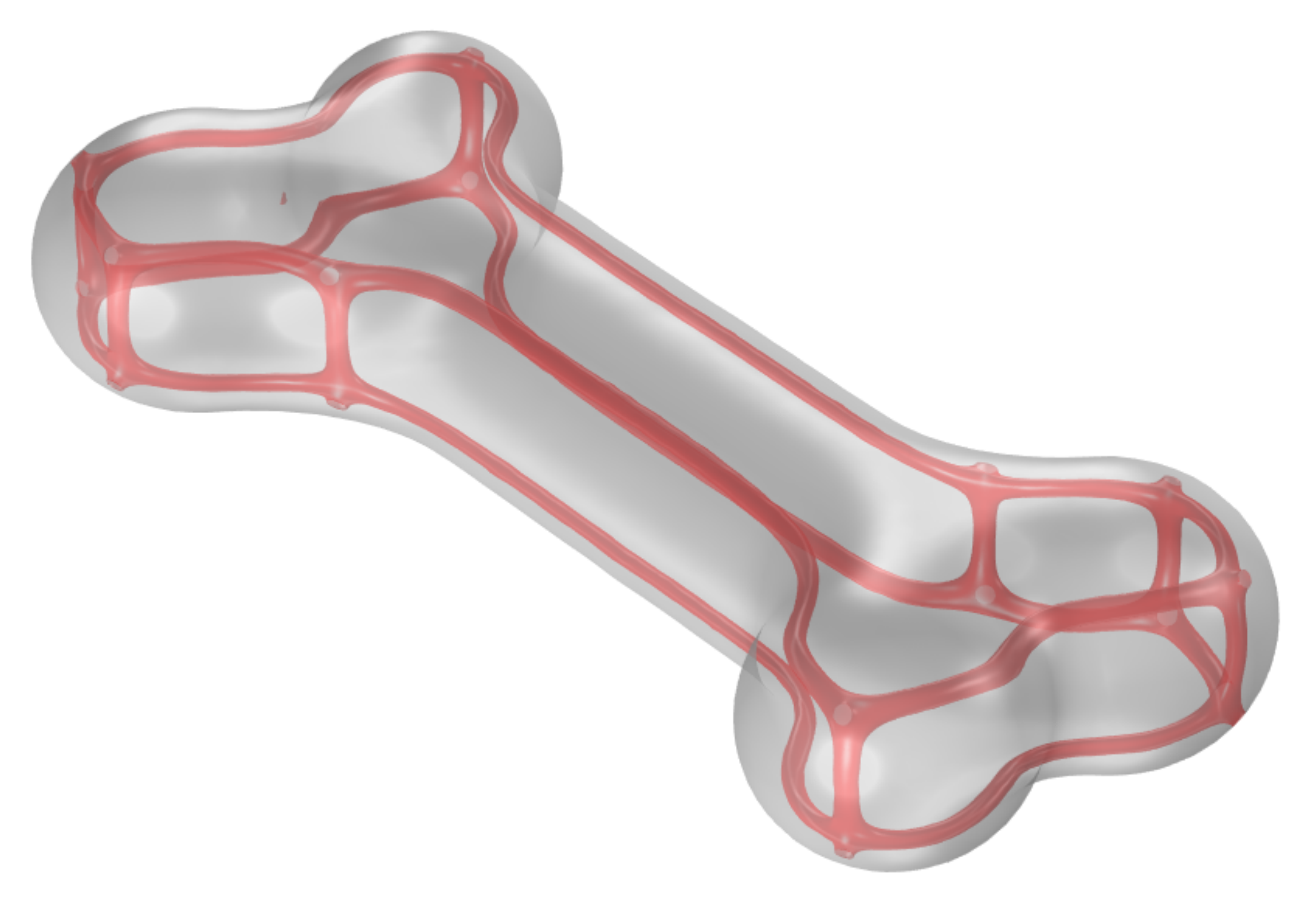}
\caption{{Solution of the Ginzburg-Landau PDE subject to the natural boundary conditions in a bone-shaped domain. Note that the singular set (marked in red) does not exhibit twisting (cf. \cite{palmer2019algebraic}).}}\label{fig:bone}
\end{center}
\end{figure}
\newpage
\section{$n$-crosses via higher order $\mathcal{Q}$-tensors}
\label{sec:ncross}

{ 
In this section we rigorously develop a tensor representation of $n$-cross fields.} In the following { discussion} we distinguish a vector as $\mathbf{a}$, a square matrix as $A$, and a tensor as $\mathcal{A}$. Let $A,B$ be two square matrices.  We denote the inner product $\left< A,B \right> = \tr (B^T A)$ that induces the norm, $|A|^2 = \left< A,A \right>$.  Finally, we will let $[A,B] = AB - BA$ and $(A,B) = AB + BA$.

Consider $n$ mutually orthogonal vectors $\mba^k \in \SP^{n-1}$ with components $a_j^k$, where $j,k=1,\ldots,n$ and the associated $n$-cross. Our main result in this section is to express the $n$-cross in terms of tensor products of $n$ projection matrices. For $n\in\N$ denote $\mathbb{M}^n$  the set of $n\times n$ symmetric matrices with real entries. Since the $n$-cross is defined by $n$ orthogonal line fields, we introduce $n$ projection matrices
\[
P^k_{ij} = (\mathbf{a}^k\otimes \mathbf{a}^k)_{ij} = a^k_i a^k_j
\]
in $\Mtr^n:=\left\{
A \in \Mbold^n : \tr A =1\right\}$ that are invariant with respect to inversions $\mba^k\to-\mba^k$. We have 
\[
P^i P^j=(\mba^i\otimes\mba^i)(\mba^j\otimes\mba^j)=\delta_{ij}(\mba^i\otimes\mba^j),\]
so that
\[ \left< P^i,P^j \right> =\mathrm{tr}\left( (P^i)^T P^j\right)=\mathrm{tr}\left(P^i P^j\right)=\delta_{ij}.\]

An $n$-cross can equivalently be defined as an unordered $n$-tuple of projection matrices $P^k$, $k=1,\ldots,n$. Thus we would like to define a mathematical object that incorporates all $P^k$, $k=1,\ldots,n$ and is invariant with respect to permutations of $P^j$ and $P^k$ for all $j,k=1,\ldots,n$. Clearly, { $\sum_{i=1}^n P^i$ is one possible candidate, however,
\[\sum_{i=1}^n P^i=I,\]
where} $I\in \Mbold^n$ is an identity matrix, hence this sum contains no information about a particular $n$-cross and a higher-order quantity is thus needed. Similar to how we used tensor products to generate elements of the projective space $\mathbb{RP}^n$ from vectors in $\SP^{n-1}$, we now use products of projection matrices to obtain higher order tensors in $\Mtr^n \otimes \Mtr^n \subset \Mbold^{n^2}$. We 
can think about a tensor of this type in a number of different ways.  Here we will interpret it as a matrix of matrices and define the tensor product of two matrices as 
\begin{align*}
\mathcal{Q}^k_{ij}  = (P^k\otimes P^k)_{ij} = a^k_i a^k_j P^k \\
\end{align*}
with elements 
\begin{equation} 
\label{eq:Qterm}
\mathcal{Q}^k_{ijrs}  = (P^k\otimes P^k)_{ijrs} 
 = \LC (\mathbf{a}^k\otimes \mathbf{a}^k) \otimes (\mathbf{a}^k\otimes \mathbf{a}^k) \RC_{ijrs} 
 = a^k_i a^k_j a^k_r a^k_s
\end{equation}
and the associated product for which the blocks satisfy
\[(\mathcal{Q}\mathcal{R})_{ij}=\sum_{k=1}^n \mathcal{Q}_{ik} \mathcal{R}_{kj}.\]
Whenever convenient, we will also think of the same tensor as an element $Q \in \Mbold^{n^2}$:
\[
Q^k_{pq} = \mathcal{Q}^k_{ijrs} \qquad p= i\,(n-1)+r \hbox{ and } q=j\,(n-1)+s.
\]
associated with the standard matrix product in $\Mbold^{n^2}$.

We now define the object $\mathcal{Q}$ representing the $n$-cross as the sum of $\mathcal{Q}^k$ over the $n$ directions, that is
\begin{equation}
\label{e:Qdef}
\mathcal{Q} = \sum_{k=1}^n \mathcal{Q}^k,
\end{equation}
or, equivalently,
\begin{equation}
\label{e:Qsubdef}
\mathcal{Q}_{ij}  = \sum_{k=1}^n a^k_i a^k_j P^k, 
\end{equation}
or
\begin{equation}
\label{e:Qelementsdef}
\mathcal{Q}_{ijrs}  = \sum_{k=1}^n a^k_i a^k_j a^k_r a^k_s.
\end{equation}
By construction, $\mathcal Q$ clearly has the symmetries of the $n$-cross: it is invariant with respect to inversions and permutations of the frame vectors ${\bf a}_j$.

We can prove several important, albeit simple, results that arise from the construction of the tensor $\mathcal{Q}$.

\begin{lem} \label{lem:ProjTen}
$\mathcal{Q}^2 = \mathcal{Q}$
\end{lem}
\begin{proof}
\begin{align*}
    \LC \mathcal{Q} \RC^2_{ij}
    & = \sum_{k=1}^n \mathcal{Q}_{ik} \mathcal{Q}_{kj} \\
    & = \sum_{k=1}^n \LC \sum_{\ell=1}^n a^\ell_i a^\ell_k P^\ell \RC \LC \sum_{m=1}^n a^m_k a^m_j P^m \RC \\
    & = \sum_{k=1}^n \sum_{\ell=1}^n \sum_{m=1}^n a^\ell_i a^\ell_k a^m_k a^m_j P^\ell P^m \\
    & = \sum_{k=1}^n \sum_{\ell=1}^n \sum_{m=1}^n a^\ell_i a^\ell_k a^m_k a^m_j \LC \mba^\ell \otimes \mba^\ell \RC  \LC \mba^m \otimes \mba^m \RC \\
    & =  \sum_{\ell=1}^n \sum_{m=1}^n a^\ell_i a^m_j \LC \sum_{k=1}^n a^\ell_k a^m_k \RC \LC \mba^\ell \otimes \mba^m \RC \\
     & =  \sum_{\ell=1}^n \sum_{m=1}^n a^\ell_i a^m_j \LC \mba^\ell \otimes \mba^m \RC \delta_{\ell m} \\
     & =\sum_{m=1}^n a^m_i a^m_j \LC \mba^m \otimes \mba^m \RC  \\
    & = \mathcal{Q}_{ij}
\end{align*}
where we used $\mba^k \cdot \mba^\ell = \delta_{k \ell}$.
\end{proof}
The consequence of invariance of crosses under permutations of lines that form a cross is the following
\begin{lem} \label{lem:permute}
$\mathcal{Q}$ is a symmetric tensor. In particular, it is invariant with respect to permutations of indices:
\begin{equation}
\label{eq:Qsymm}
\mathcal{Q}_{ijrs} = \mathcal{Q}_{\sigma(ijsr)}  
\end{equation}
where $\sigma \in S_4$, the group of permutations on $\{1,2,3,4\}$.
\end{lem}
\begin{proof}
This follows immediately from the form of \eqref{e:Qelementsdef}.
\end{proof}
The remaining facts deal with submatrices of $\mathcal{Q}$.
\begin{lem} \label{lem:traceQsub}
Submatrices $\mathcal{Q}_{ij}$, $i,j=1\ldots,n$ of $\mathcal{Q}$ are symmetric and satisfy the following trace condition: 
\begin{equation}
\tr Q_{ij} = \delta_{ij}
\end{equation}
\end{lem}
\begin{proof}

The symmetry of $Q_{ij}$ immediately follows from the previous lemma. To find the trace of $Q_{ij}$, first note 
\begin{align*}
    \tr \mathcal{Q}_{ij} & = \tr \LC \sum_{\ell = 1}^n a^\ell_i a^\ell_j P^\ell \RC \\
    & = \sum_{\ell = 1}^n a^\ell_i a^\ell_j \tr P^\ell \\
    & = \sum_{\ell =1}^n a^\ell_i a^\ell_j.
\end{align*}

Next, we note that $\mba^1,\ldots, \mba^n$ form an orthonormal basis which implies the corresponding matrix $(\mba^1 | \cdots | \mba^n)$ forms an orthogonal matrix.  Since the matrix is orthogonal, the row vectors of this matrix, $\mathbf{b}^k = (a^1_k, \ldots, a^n_k)$ are an orthonormal basis.  Therefore, 
\[
\delta_{ij} = \mathbf{b}^i \cdot \mathbf{b}^j = \sum_{\ell =1}^n a^\ell_i a^\ell_j,
\]
which completes the proof.
\end{proof}


Finally, we have the following, 
\begin{lem}\label{submatrices}
	Submatrices $\mathcal{Q}_{ij}$, $i,j=1\ldots,n$ of $\mathcal{Q}$ have the common eigenframe $\left\{\mathbf{a}^k\right\}_{k=1}^n$ and, therefore, commute. 
\end{lem}
\begin{proof}
	For any $i,j,l=1,\ldots,n$, we have
	\[\mathcal{Q}_{ij}{\mathbf a}^l= \sum_{k=1}^n a^k_i a^k_j P^k {\mathbf a}^l= \left(a^l_i a^l_j\right) {\mathbf a}^l.\]
	The commutation property of $\mathcal{Q}_{ij}$, $i,j=1\ldots,n$  immediately follows.
\end{proof}
We remind the reader of the following related result that we will use in a sequel.
\begin{lem}\label{lem:commute}
If $A$ and $B$ are any two symmetric matrices in $\Mbold^n$ that commute with each other, then $A$ and $B$ have a common eigenframe.
\end{lem}
\begin{proof}
Suppose $B \mathbf{v} = \lambda \mathbf{v}$, then $BA \mathbf{v} = AB \mathbf{v} = \lambda A\mathbf{v}$.  Therefore, $A \mathbf{v}$ is an eigenvector of $B$ associated with the eigenvalue $\lambda$ and $A:\mathrm{ker}(B - \lambda I)\to \mathrm{ker}(B - \lambda I)$. Because $A$ and $B$ are symmetric, both have associated bases of orthonormal  eigenvectors in $\R^n$ that we will denote by $\left\{{\bf a}_i\right\}$ and $\left\{{\bf b}_i\right\}$, respectively. Suppose that ${\bf v}\in\mathrm{ker}(B - \lambda I)$ and the equation $A{\bf x}={\bf v}$ has a solution. Then \[BA{\bf x}=AB{\bf x}=\sum_i x_iAB{\bf b}_i=\sum_i x_i\lambda_iA{\bf b}_i=\lambda{\bf v}\in\mathrm{ker}(B - \lambda I),\] where $\lambda_i$ is an eigenvalue of $B$ corresponding to ${\bf b}_i$. It follows that $x_i=0$ for all ${\bf b}_i\notin\mathrm{ker}(B - \lambda I)$ so that ${\bf x}\in\mathrm{ker}(B - \lambda I)$. From this, we conclude that the preimage of the set $\mathrm{ker}(B - \lambda I)$ under the map $A:\R^n\to \R^n$ is $\mathrm{ker}(B - \lambda I)$, hence $\mathrm{ker}(B - \lambda I)$ is spanned by eigenvectors of $A$.

\end{proof}

\begin{rem} \label{rem:numberqs}
We can now use Lemma \ref{lem:permute} to calculate the number of unique entries in $\mathcal{Q}$. With the help of \eqref{eq:Qterm} we can see that this number must be the same as the dimension of the space of polynomials of degree four in $n$ variables, or
\[\left(\begin{array}{c}
     n+3  \\
     n-1 
\end{array}\right)\]
Now, accounting for symmetry, there are $n(n+1)/2$ distinct $n\times n$ submatrices comprising $\mathcal{Q}$. It follows that Lemma \ref{lem:traceQsub} gives $n(n+1)/2$ additional linear constraints on the components of $\mathcal{Q}$. We conclude that the number of unique entries in $\mathcal{Q}$ is 
\[\left(\begin{array}{c}
     n+3  \\
     n-1 
\end{array}\right)-\frac{n(n+1)}{2}=\frac{n\left(n^2-1\right)(n+6)}{24}.\]
\end{rem}

{ 
\subsection{Cross-fields via zero sets of polynomials}
\label{sec:ZSP}

The construction we have just described can be summarized as follows. Within our framework the set of $n$-crosses is
\begin{equation}\label{eq:mfld_first_def}
\Mcross = \{{\mathcal Q}\in \Mbold^{n^2} \,\,\,\mbox{is a symmetric tensor}, \,\,\,  {\mathcal Q}^2 = {\mathcal Q},  \,\,\, {\rm tr}({\mathcal Q}_{ij}) = \delta_{ij}\},
\end{equation}
where \emph{symmetric tensor} refers to the property established in Lemma~\ref{lem:permute} or, equivalently, to invariance of the tensor under the permutation operators defined in the Appendix.  In Section~\ref{sec:GLn}, we show that there is a one-to-one correspondence between $n$-crosses and tensors in ${\mathcal Q}\in \Mcross$ so that, in particular, a unique $n$-cross can be recovered from a tensor ${\mathcal Q}\in \Mcross$. Further, from \eqref{eq:mfld_first_def} we have that  $\Mcross$ can be defined as the zero set of a finite family of polynomials in the Euclidean space $\Mbold^{n^2}$.  In this case the polynomials are either quadratic (arising from ${\mathcal Q}^2 = {\mathcal Q}$) or linear (arising from the symmetry or trace conditions in \eqref{eq:mfld_first_def}).  Therefore, by definition, $\Mcross$ is an algebraic variety.

We arrived at the set $\Mcross$ while searching for convenient subsets of the Euclidean space that faithfully represent the quotient space
$SO(n) / O_n$. Here $SO(n)$ denotes the set of orthogonal $n\times n$ matrices with real entries and the determinant equal $1$. $O_n$ is the finite subgroup of $SO(n)$ composed of the symmetries of the unit cube in $\mathbb{R}^n$; it is the classical octahedral group when $n=3$.  The quotient space $SO(n) / O_n$ captures the symmetries we require of an $n$-cross.  

The idea of defining $3$-crosses as elements of the quotient space $SO(3)/O_3$ and describing this space as a subset of a Euclidean space is not new. In \cite{Huang:2011:BAS:2070781.2024177, palmer2019algebraic,raysokolov} the elements of $SO(3)/O_3$ are represented by polynomials in three variables of the form
$$
p(x) = p_0(R^Tx)
$$
or, more precisely, by the restriction of these polynomials to the unit sphere $S^2\subset \mathbb{R}^3$.  Here $x \in \mathbb{R}^3$, the $3\times 3$ matrix $R$ is orthogonal, and $p_0$ is a specific homogeneous polynomial of degree $4$ in three variables.  The polynomial $p_0$ is chosen so that $p_0(Rx) = p_0(x)$ for all $x\in \mathbb{R}^3$ exactly when $R\in O_3$, i.e., $p_0$ is invariant under the action of the octahedral group $O_3$.  By projecting the above polynomials onto the set of harmonic, homogeneous polynomials of degree $4$ (also known as band $4$ spherical harmonics), the authors of \cite{Huang:2011:BAS:2070781.2024177, palmer2019algebraic,raysokolov} obtain a subset of $\mathbb{R}^9$ that is diffeomorphic to $\Mcross$.  This subset can also be expressed as the zero set of a family of quadratic polynomials, in this case in $\mathbb{R}^9$.  Finally in this construction, alignment of a $3$-cross field on the boundary of a domain is characterized in terms of coefficients of a polynomial with respect to a basis of spherical harmonics (cf. \cite{palmer2019algebraic}).

 In a related work \cite{Schaft}, the authors appeal to the equivalence between homogeneous polynomials of degree $4$ and $4^{th}$-order symmetric tensors to obtain a representation of 3-cross fields as fourth order symmetric tensors.  In particular, they show that fourth-order symmetric tensors that are projections satisfying some additional affine constraints encode a recoverable 3-cross fields.  This is the same framework we choose, except our motivation does not exploit the equivalence between homogeneous polynomials of degree $4$ and $4^{th}$-order symmetric tensors. While it was not discussed in \cite{Schaft}, alignment of a $3$-cross field on the boundary of a domain is characterized in our Proposition \ref{prop:being_part_of_frame} in terms of either commutators of matrices or their eigenvectors.

Now define a $3$-cross field on a domain $\Omega \subset \mathbb{R}^3$ as a map
$$
u : \Omega \to \Mcross.
$$
To such a map we can assign an energy, for example given by the Dirichlet integral
$$
E(u) = \frac{1}{2} \int_\Omega \abs{\nabla u}^2 dx.
$$
Then an optimal cross field can be defined as a minimizer of $E$ among maps $u : \Omega \to \Mcross$ that satisfy boundary conditions on $\partial\Omega$ enforcing a particular choice of boundary alignment.  We shall see in the next section that such maps necessarily have singularities even in simple geometries like a ball in $\mathbb{R}^3$. In some instances (cf. Fig.~\ref{fig0} and \ref{fig:bone}) this has the consequence that the Dirichlet integral $E(u)$ is infinite; this naturally leads to considering relaxation schemes for $E$ that spread these singularities.  In \cite{palmer2019algebraic} this by accomplished by adding the distance to $\Mcross$ as a penalty term to the energy $E$. The relaxed energy is actually never explicitly used in \cite{palmer2019algebraic} as they drive admissible maps toward an equilibrium using an MBO algorithm based on advancing the solution via heat flow and then projecting it onto the target manifold on each time step. In contrast, we propose a relaxation scheme (Section~\ref{sec:GLn}) that uses one of the polynomials that define $\Mcross$ as the penalty term. Our evolution algorithm is the gradient flow for the relaxed energy that proceeds along maps that are near but not necessarily in $\Mcross$. Whether or not our algorithm and that of \cite{palmer2019algebraic} evolve toward the same equilibrium is an interesting open question.


{ To further illuminate the differences in these relaxation schemes, we detail more of the approach in \cite{palmer2019algebraic}.
Let 
\begin{equation}
\label{eq:projsymtensor}
\mathbb{P}_{Sym} \mathcal{M}
= \frac{1}{|P_k|} \sum_{\sigma \in P_k} T_{\sigma}\mathcal{M}
\end{equation}
for a $k$-th order tensor, where $P_k$ set of permutations $\sigma$ of length $k$. This operator defines orthogonal projection onto the set of symmetric $k$-tensors. As a substitute to $\Mcross$, one can now define the larger {\em odeco variety} as 
\begin{equation} \label{eq:odecodef}
\mathbb{M}^n_{\text{odeco}} 
= \{ \mathcal{M} \in \text{Symmetric 4-tensors such that }  (\mathbb{I} - \mathbb{P}_{Sym} ) (\mathcal{M}^2) =  (\mathbb{I} - \mathbb{P}_{Sym} ) (\mathcal{M}^4) = 0 \},
\end{equation}
see Corollary~\ref{cor:ODECO} in the Appendix.
Given that $\Mcross\subset \mathbb{M}^n_{\text{odeco}}$, one can work with the odeco variety and its relaxations instead of $\Mcross$.  
One advantage of using odeco variety is that it is strictly larger than $\Mcross$, hence some of the singularities of cross fields may be avoided by working with the odeco variety instead.  To generate the orthogonal coordinate system, the authors run an MBO scheme with the projection step onto either odeco varieties or 3-cross fields. 

Consideration of more general cross fields ---such as odeco---represent future avenues of exploration, see Section~\ref{sec:discussion}.
}
}

\section{$n$-crosses conforming to the boundary of an $n$-dimensional domain}
 \label{sec:boundary}
 
In this section we discuss the proper way of prescribing an $n$-cross field on the boundary of an $n$-dimensional domain (or, more generally on an $n-1$-dimensional Lipschitz manifold). In particular, we will focus on describing what can be thought of as the natural boundary conditions for the Ginzburg-Landau variational problem that we will consider below. Here we require that the $n$-cross field at every point on the boundary contain a line that is parallel to the normal to the boundary.  This condition can be phrased in a few equivalent ways, which are presented in Proposition \ref{prop:being_part_of_frame}.  Finally, cross fields generate singularities on two dimensional boundaries, see for example \cite{fogg, palacios,Ray:2008:NDF:1356682.1356683}.  At the end of this section we provide a simple proof of this for cross fields on a ball in $\mathbb{R}^3$.  We also give examples of $3$-cross fields on a ball in $\mathbb{R}^3$ with what can be thought of as the simplest possible configurations of necessary singularities.

Let us start by recalling that we write ${\mathcal Q}\in \Mbold^{n^2}$ as
$$
{\mathcal Q} = \left ( \begin{array}{ccc} {\mathcal Q}_{11} & \cdots & {\mathcal Q}_{1n} \\ \vdots & \ddots & \vdots \\ {\mathcal Q}_{n1} & \cdots & {\mathcal Q}_{nn} \end{array}  \right ),
$$
{ where} each ${\mathcal Q}_{ij}\in \Mbold^{n}$.  By Theorem \ref{thm:limitrelax}, we know that each ${\mathcal Q}\in \Mcross$ as above, has an associated $n$-cross.  Let us recall here that this is the set of unordered rank one, orthogonal projections $P^k \in \Mtr^n$ defined by an orthonormal basis $\{ {\bf a}^k \}_{k=1}^n$, which in turn is determined by ${\mathcal Q}$ up to order.   In particular we have
$$
P^j = {\bf a}^j \otimes {\bf a}^j,
$$
and
\begin{equation}\label{basic_entry}
{\mathcal Q}_{ij} =  \sum_{k=1}^3 {{\bf a}}^k_i{{\bf a}}^k_jP^k.
\end{equation}
Let us also recall that the $n$-cross satisfies $P^jP^k = P^kP^j = \delta_{jk}P^k$ for all $j, k=1, ..., n$.  In particular, the $P^k$ commute with each other.

Our main boundary requirement will be $\nu(x)$, the normal at $x\in \partial \Omega$, be part of the frame associated to $\mathcal{Q}(x)$.  As this is an issue between a single projection matrix $P$ and a tensor $\mathcal Q$, we drop the dependence on $x$, and suppose for concreteness that $\mathbf{\nu} \in \mathbb{S}^{n-1}$ and that $P$ projects onto the subspace generated by $\mathbf{\nu}$.  It is easy to see from the discussion above that if $\mathcal{Q}\in \Mcross$ and $P$ is an element of the $n$-cross of $\mathcal{Q}$, then $\mathbf{\nu} = (\mathbf{\nu}_1, ..., \mathbf{\nu}_n)$ is an eigenvector of every $n\times n$-block $\mathcal{Q}_{ij}$ with eigenvalue ${\bf \nu}_i\mathbf{\nu}_j$.  In other words,
\begin{equation}
\label{eq:bc}
\mathcal{Q}_{ij} \mathbf{\nu}=\mathbf{\nu}_j\mathbf{\nu}_k \mathbf{\nu}
\end{equation}
for every $i,j=1,\ldots,n$ on $\partial \Omega$.  We shall see that this condition is in fact equivalent to the membership of $P$ to the $n$-cross of $\mathcal{Q}$, and to a third condition.  
\begin{proposition}\label{prop:being_part_of_frame}
Let ${\mathcal Q}\in \Mcross$, let $P$ be a fixed $n\times n$, rank 1, orthogonal projection, and let $\mathbf{\nu}\in \mathbb{S}^{n-1}$ be a unit vector in the image of $P$.  Finally, denote
$$
{\mathcal P} = \left ( \begin{array}{ccc} P_{11}P & \cdots & P_{1n}P \\ \vdots & \ddots & \vdots \\ P_{n1}P & \cdots & P_{nn}P \end{array}  \right ) = \left ( \begin{array}{ccc} \mathbf{\nu}_1\mathbf{\nu}_1P & \cdots & \mathbf{\nu}_1\mathbf{\nu}_n P \\ \vdots & \ddots & \vdots \\ \mathbf{\nu}_n\mathbf{\nu}_1P & \cdots & \mathbf{\nu}_n\mathbf{\nu}_nP \end{array}  \right ),
$$
where the $\mathbf{\nu}_i$ are the coordinates of $\mathbf{\nu}$.  The following are equivalent:
\begin{enumerate}
\item Either $\nu$ or $-\nu$ is part of the $n$-cross of $\mathcal Q$.
\item $[{\mathcal Q}, {\mathcal P}] = 0$.
\item $ {\mathcal Q}_{ij} \nu = \mathbf{\nu}_i\mathbf{\nu}_j\mathbf{\nu}$ for each $i, j = 1, ..., n$.
\end{enumerate}
\end{proposition}

\begin{proof}
Let us start by observing that (1) easily implies (2), and that (1) implies (3) by Lemma \ref{submatrices}.

We show now that (2) implies (1).  First, a direct multiplication of matrices shows that
$$
({\mathcal P}{\mathcal Q})_{ij} = \sum_{k=1}^n (PP^k)_{ij}PP^k.
$$
This shows that the condition  $[{\mathcal P}, {\mathcal Q}]=0$ implies
$$
\sum_{k=1}^n (PP^k)_{ij}PP^k = \sum_{k=1}^n (P^kP)_{ij}P^kP
$$
for every $i, j=1, ..., n$.  We now take any matrix $A\in \Mtr^n$ with entries $a_{ij}$, multiply the last identity by $a_{ij}$ and add in $i$, $j$ to obtain
$$
\sum_{k=1}^n \langle PP^k, A\rangle PP^k = \sum_{k=1}^n \langle P^kP, A\rangle P^kP.
$$
Since $A$, $P$ and $P^k$ are all symmetric, $\langle PP^k, A\rangle = \langle P^kP, A\rangle$, so we conclude that
$$
\sum_{k=1}^n \langle PP^k, A\rangle [P, P^k] = 0
$$
for every $A\in \Mtr^n$.  It is easy to see that
$$
\langle PP^k, P^j\rangle = \delta_{k, j}\langle P, P^k\rangle,
$$
so replacing $A$ by $P^j$ in the next to last equation we obtain
$$
0 = \langle P, P^j\rangle [P, P^j]
$$
for every $j=1, ..., n$.  Since the $P^j$ and $P$ are all orthogonal projections of rank $1$, it is easy to conclude from here that $P$ is indeed one of the $P^j$.
 
We show last that condition (3) also implies (1).  To do this we observe that clearly (3) implies that
$$
[{\mathcal Q}_{ij}, P] = 0
$$
for every $i, j=1, ..., n$, because ${\mathcal Q}\in \Mcross$. Since
$$
{\mathcal Q}_{ij} = \sum_{k=1}^n P_{ij}^k P^k,
$$
this implies that
$$
[P, {\mathcal Q}_{ij}] = \sum_{k=1}^n P_{ij}^k [P, P^k] = 0
$$
for every $i, j = 1, ..., n$.  This clearly implies that $[P, P^k]=0$ for $k=1, ..., n$, so again, $P$ is one of the $P^k$.
\end{proof}

Next, we record a simple relation between topologically trivial maps
${\mathcal Q}:\partial \Omega \to \Mcross$ that always contain $P_\mathbf{\nu}$ as part of their $n$-cross, and tangent vector fields on $\partial \Omega$.  

For this we first consider a map $\mathcal Q$ defined on $\partial \Omega \setminus V$, where $V\subset \partial \Omega$ is some finite subset of the boundary, possibly empty.  Denoting by $\pi_1(A)$ the fundamental group of $A$, by topologically trivial we mean that the image of $\pi_1(\Omega\setminus V)$ by the map induced by $u$ on fundamental groups, is the identity element of $\pi_1(\Mcross)$.  For this situation we have the
\begin{proposition}\label{prop:frame_and_tangents}
Let $n\geq 3$, $\Omega \subset \R^{n}$, and $V \subset \partial \Omega$ be a finite subset of isolated points, possibly empty.  For every smooth map ${\mathcal Q}:\partial \Omega \setminus V \to \Mcross$ that is topologically trivial in the sense described above, and that always contains $\mathbf{\nu}$ as part of its frame, there are $(n-1)$ smooth, unit, tangent vector fields
$$
{\bf \tau}_{\mathcal Q}^j:\partial \Omega \setminus V\to \mathbb{S}^{n-1}, j=1,..., n-1,
$$
that are also part of the $n$-frame of $\mathcal Q$, and that along with ${\bf \nu}$ form an orthonormal basis of $\R^{n}$.  Conversely, given $(n-1)$  such unit, tangent vector fields $\mathbf{\tau}^j :\Omega\setminus V\to \mathbb{S}^{n-1}$, $j=1, ..., n-1$, there is a map ${\mathcal Q}:\partial \Omega \setminus V\to \Mcross$ that has ${\bf \nu}$ in its $n$-cross, as well as the ${\bf \tau}_j$.
\end{proposition}

\begin{proof}
The converse part of the proposition is essentially trivial so we concentrate on the direct implication.  We first recall that $SO(n)$ is a covering space for $\Mcross$, although not the universal cover of $\Mcross$.  Still, if $P^1_0$, ..., $P^n_0$ are the projections onto the spaces generated by each of the vectors of some fixed canonical basis, then
\begin{align*}
T:&SO(n)\to \Mcross \\
&R \to T(R) = \sum_{k=1}^n {\bf X}_{RP^k_0R^T}\otimes {\bf X}_{RP^k_0R^T}
\end{align*}
is a covering map.  Here we use notation of Section 8, in particular the isomorphism ${\bf X}:\Mall \to \R^{n^2}$ between the set of $\Mall$ of all $n\times n$ matrices and $\R^{n^2}$ defined in \ref{e:define_bold_X}.

The condition that ${\mathcal Q}:\partial \Omega \setminus V\to \Mcross$ be topologically trivial  is known to guarantee that $\mathcal Q$ lifts through
$$
R:\partial \Omega \setminus V \to SO(n).
$$
This means
$$
{\mathcal Q}(x) = T(R(x)) = \sum_{k=1}^n {\bf X}_{R(x)P^k_0R(x)^T}\otimes {\bf X}_{R(x)P^k_0R(x)^T}.
$$
Now we assume that $P_{\bf \nu}(x)$ is part of the $n$-cross of ${\mathcal Q}(x)$ at every $x\in \partial \Omega \setminus V$.  The same arguments we used in our previous lemma show that at every $x \in \partial \Omega \setminus V$, $P_{\bf \nu}(x)$ is one of the $R(x)P^k_0R(x)^T$.  Since $V$ is a finite set, and $n\geq 3$, $\partial \Omega \setminus V$ is connected.  This implies that $P_{\bf \nu}(x)$ is one of the $R(x)P^k_0R(x)^T$ with the same $k$ for every $x\in \partial \Omega \setminus V$.  Without loss of generality assume $k=1$.  Calling ${\bf e}^1_0$, ..., , ${\bf e}^n_0$ the canonical basis behind $P^1_0$, ..., $P^n_0$, clearly $R(x){\bf e}^2_0$, ..., $R(x){\bf e}^n_0$ are both smooth, unit, tangent vector fields on $\partial \Omega \setminus V$.
\end{proof}

The second aspect we will consider in this section stems from the fact that there are topological obstructions to the existence of smooth maps that satisfy the boundary conditions we describe here.  Because of this, in order to build boundary maps that satisfy our boundary conditions, one is forced to introduce singularities on the boundary.  We will give a simple criterion that allows us to build boundary maps with a finite number of point singularities.

Once we have the previous Proposition we can use some classical facts regarding tangent vector fields to draw conclusions relevant to our situation.  The first is the following consequence of the Poincare Hopf theorem:

\begin{corollary}\label{cor:no_smooth_maps}
Let $\Omega = B_R(0)$ be the ball of radius $R>0$ around the origin on $\R^3$.  There is no smooth map ${\mathcal Q}:\partial \Omega \to \Mcrossthree$ that contains either ${\bf \nu}(x)$ or $-{\bf \nu}(x)$ as part of its frame at every $x\in \partial \Omega$.
\end{corollary}
Another use of Proposition \ref{prop:frame_and_tangents} is the following:  the Poincare Hopf Theorem tells us not only that any tangent vector field to $\mathbb{S}^2$ must have zeros, but also that the sum of the degrees of the zeros of any tangent vector field to $\mathbb{S}^2$ must equal the Euler characteristic of the sphere.  For $\mathbb{S}^2$ (and also for $\mathbb{S}^n$, even $n$), the Euler characteristic is $2$.  The simplest possible combination of zeros and degrees under this constraint is one zero with degree two.

{We now give an example of a frame field in $\Omega = B_R(0)\subset \R^3$ that contain ${\bf \nu}$ in its frame at all but one point on $\partial \Omega$.}  Denote by ${\bf p}_s$ the south pole of $\partial B_R(0)$, pick ${\bf e} \in \mathbb{S}^2$ such that ${\bf p}_s\cdot {\bf e}=0$ and let
$$
\hat{\bf r}_s(x)=\frac{x-{\bf p}_s}{\abs{x-{\bf p}_s}}.
$$
With this define
\begin{equation}
\label{eq:a1}
{\bf a}^1(x) = R^{-1}({\bf p}_s - 2({\bf p}_s\cdot \hat{\bf r}_s(x)) \hat{\bf r}_s(x)),
\end{equation}
\begin{equation}
\label{eq:a2}
{\bf a}^2(x) = {\bf e} - 2({\bf e}\cdot \hat{\bf r}_s(x)) \hat{\bf r}_s(x),
\end{equation}
and
\begin{equation}
\label{eq:a3}
{\bf a}^3(x) = {\bf a}^1(x) \times {\bf a}^2(x).
\end{equation}
Direct computations show that this is an orthonormal frame at every $x\in \overline{B_R(0)}\setminus \{{\bf p}_s\}$.  Furthermore, whenever $x\in \partial B_R(0)\setminus \{{\bf p}_s\}$, we have both that ${\bf a}^1(x) = \frac{x}{\abs{x}} = {\bf \nu}(x)$, and that the vector field ${\bf a}^2$ is the image through the (inverse of the) stereographic projection from the south pole of the vector field that differentiates with respect to one of the coordinates on the complex plane. {In Fig. \ref{fig:ex1} we illustrate the distribution of the cross-field $\left\{{\bf a}^1,{\bf a}^2,{\bf a}^3\right\}$ in $B_R(0)$.}
\begin{figure}
\begin{center}
\includegraphics[scale=.6]{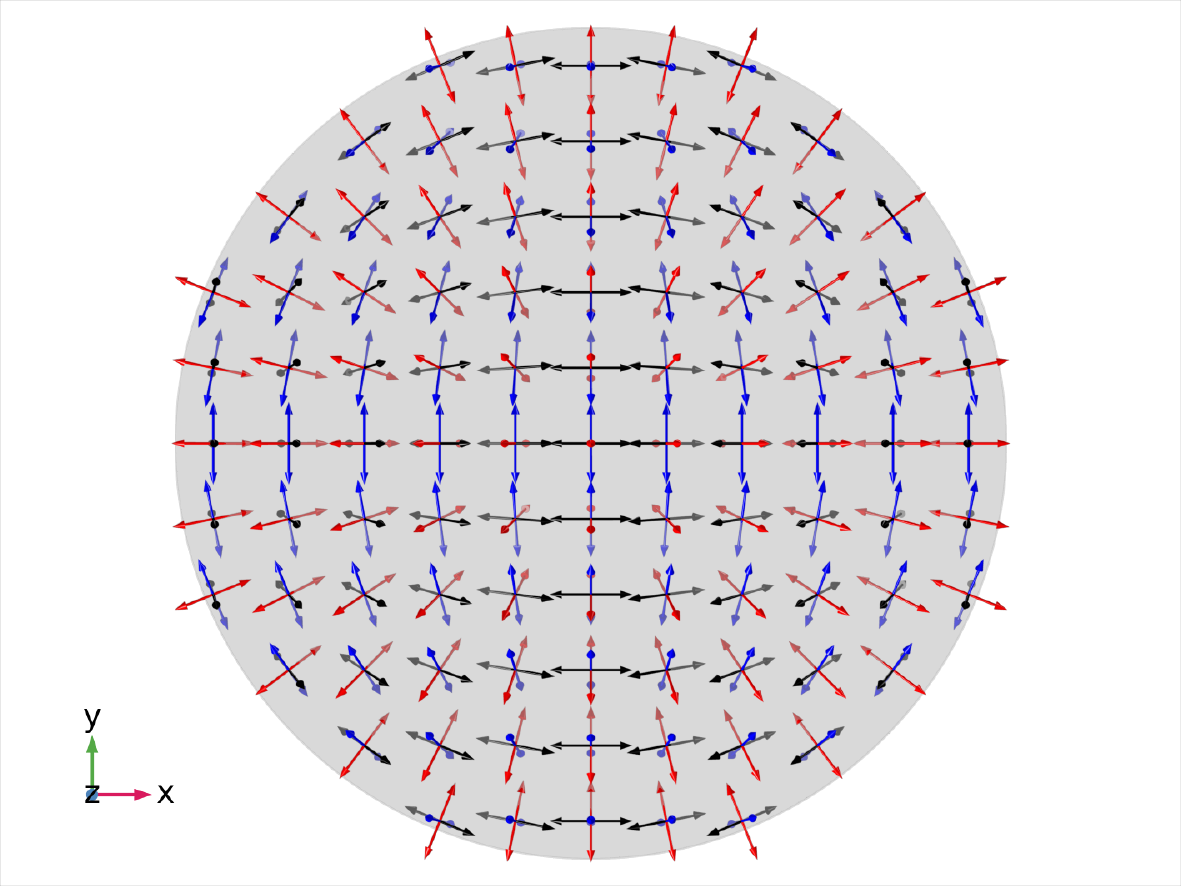} \qquad \includegraphics[scale=.6]{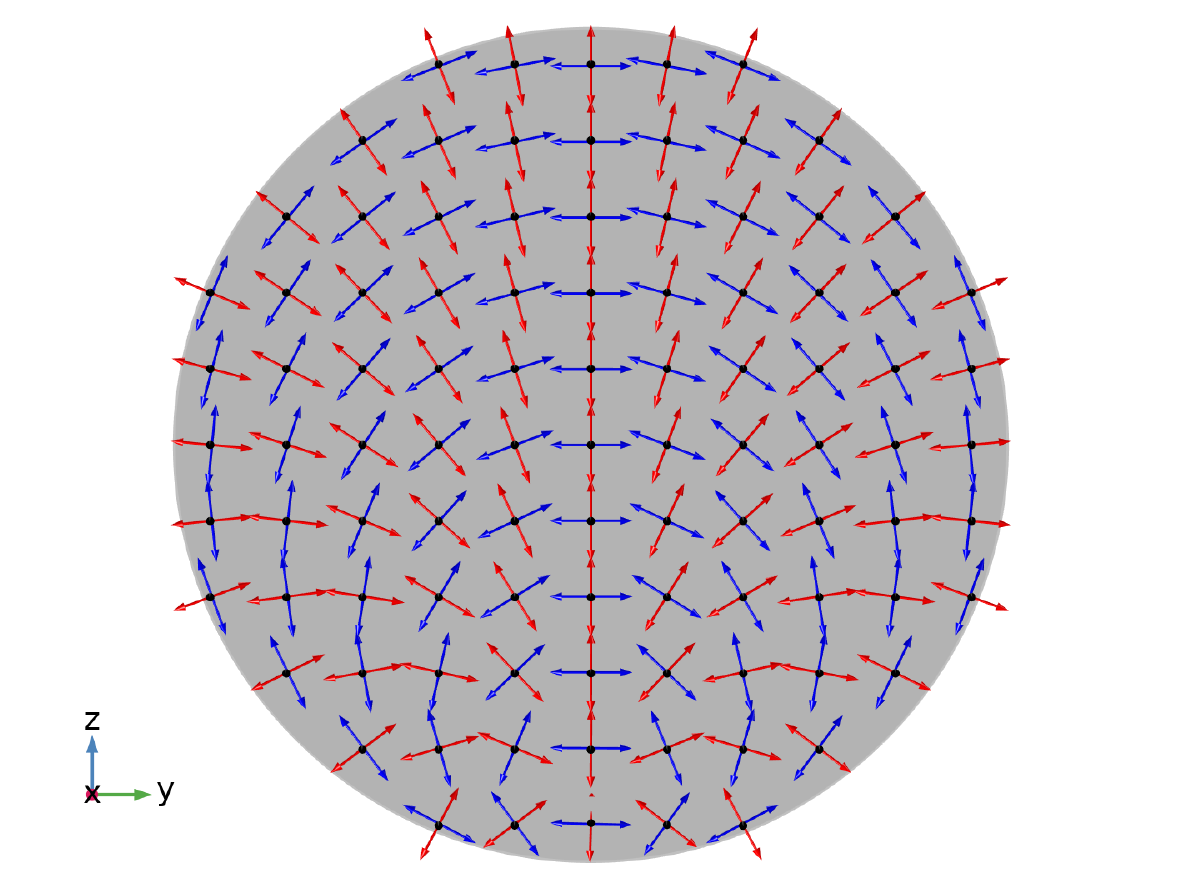} \\ \vspace{5mm} \includegraphics[scale=.7]{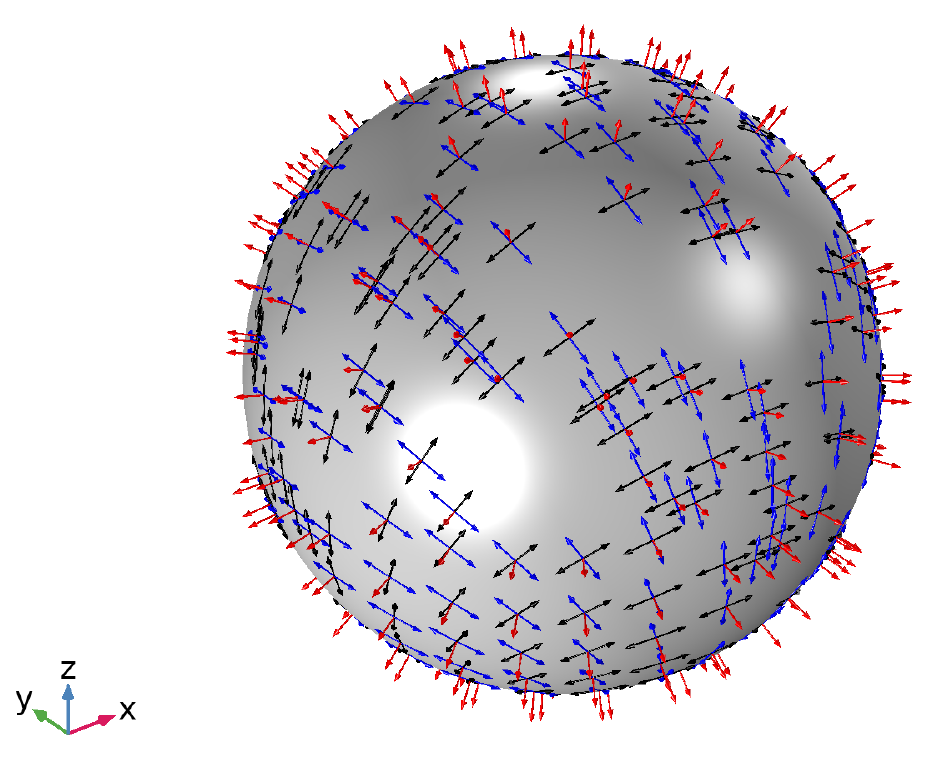}
\caption{{Cross-sections of the $3$-cross field \eqref{eq:a1}-\eqref{eq:a3} along $xy$- and $yz$-planes, respectively (top row). The $3$-cross field \eqref{eq:a1}-\eqref{eq:a3} on the surface of the sphere (bottom row). The singular point is located at the south pole of the sphere. The vectors ${\bf a}^1$, ${\bf a}^2$, and ${\bf a}^3$ are marked in different colors to aid visualization.}}\label{fig:ex1}
\end{center}
\end{figure}

{Another prototypical situation is that of a 3-cross field in an infinitely long cylinder that contains the normal ${\bf \nu}$ in its frame on the boundary. The next example shows that this cross field does not need to contain a singularity in the interior of the cylinder due to the so-called "escape" phenomenon, well-known in the literature on nematic liquid crystals. Indeed, consider $D_R(0)$---a cross-section of a circular cylinder of radius $R$ by a plane perpendicular to its axis and let $(r, \theta)$ be the standard set of polar coordinates in $\R^2$. Define $\alpha(r)= \frac{\pi r}{2R}$ and set
\begin{equation}
\label{eq:a1.1}
{\bf a}^1(x) =(\cos^2\theta\cos\alpha(r)+\sin^2\theta, \sin{\theta}\cos{\theta}(\cos\alpha(r)-1), -\cos{\theta}\sin\alpha(r)),
\end{equation}
\begin{equation}
\label{eq:a2.1}
{\bf a}^2(x) =(\sin{\theta}\cos{\theta}(\cos\alpha(r)-1),\sin^2\theta\cos\alpha(r)+\cos^2\theta, -\sin{\theta}\sin\alpha(r)),
\end{equation}
and
\begin{equation}
\label{eq:a3.1}
{\bf a}^3(x) =(\cos{\theta}\sin\alpha(r), \sin{\theta}\sin\alpha(r),\cos{\alpha(r)}).
\end{equation}
It is easy to check that this is indeed an orthonormal frame at every $x\in D_R(0)$, that ${\bf a}^1(x) = {\bf \nu}(x) = \frac{x}{\abs{x}}$ when $x\in \partial D_R(0)$,  and that this field is smooth in the interior of $D_R(0)$. The corresponding frame field is depicted in Fig. \ref{fig:ex2}.}
\begin{figure}
\begin{center}
\includegraphics[scale=.7]{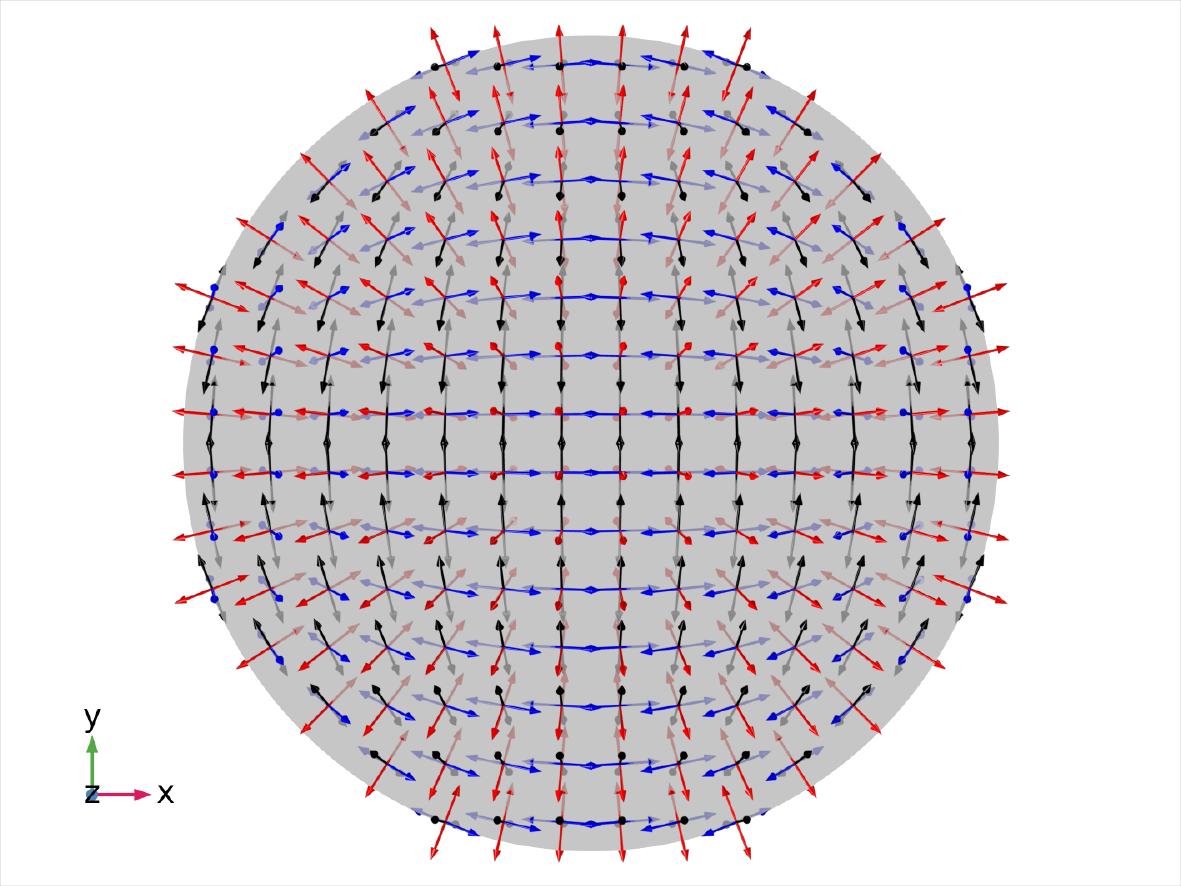}
\caption{{The "escaped" $3$-cross field \eqref{eq:a1.1}-\eqref{eq:a3.1} in the cylinder. The vectors ${\bf a}^1$, ${\bf a}^2$, and ${\bf a}^3$ are marked in different colors to aid visualization. The vector field ${\bf a}^3$ normal to the boundary, avoids singularity by escaping into the third dimension, i.e., reorienting along the $z$-axis at the center of the cross-section}}\label{fig:ex2}
\end{center}
\end{figure}

\begin{rem}
We remark that neither of these examples of the frame fields have interior singularities.  Further, for both  examples, the energies we consider in \eqref{eq:GLenergy} and \eqref{eq:GLenergyWA} have finite values independent of $\varepsilon$.  More precisely, the energies have finite contributions from the respective gradient terms and zero contributions from the potential and the penalty on the boundary.
\end{rem}

\section{Ginzburg-Landau relaxation and recovery of the $n$-cross field}
\label{sec:GLn}

We first define our ambient manifold and then define the relaxation procedure to the $n$-cross. Our relaxation will start from the set of symmetric tensors with certain trace conditions on its submatrices.   This is a similar definition to one found in \cite{Schaft}:
\begin{definition}
Set
\begin{equation} \label{eq:mrelaxdef}
\Mrelax = \{{\mathcal Q} \in \Mbold^{n^2}:  \mathcal Q_{ijk\ell} = \mathcal{Q}_{\sigma(ijk\ell)} \,\,\,\,\mbox{for all}\,\,\,\,\sigma \in S_4, \,\,\,\, {\rm tr}({\mathcal Q}_{ij}) = \delta_{ij} \}.
\end{equation}
We also describe the subset of elements of this space that are projections:
\begin{equation} \label{eq:mframedef}
\Mcross = \{{\mathcal Q} \in \Mrelax: {\mathcal Q}^2 = {\mathcal Q}\}.
\end{equation}
\end{definition}

Our main result in this section is the following theorem which shows that elements of $\Mcross$ are in fact $n$-cross fields.  
\begin{theorem} \label{thm:limitrelax}
For every ${\mathcal Q}\in \Mcross$ there are $n$ rank-1, orthogonal projection matrices with pairwise perpendicular images $P^1, ..., P^n \in \Mtr^{n}$ such that
$$
{\mathcal Q} = \sum_{j=1}^n P^j \otimes P^j.
$$
In other words, for every ${\mathcal Q}\in\Mcross$ there are matrices $P^1, ..., P^n \in \Mtr^n$ such that
\[
I = \mathop{\sum}\limits_{j=1}^n P^j
\]
where
\[
(P^j)^2 = (P^j)^T = P^j, \, {\rm tr}(P^j) = 1, \, 
P^j P^k = P^k P^j = \delta_{k,j}P^j,
\]
for all $j, k = 1, ...$, and
\[
{\mathcal Q} = \sum_{j=1}^n P^j \otimes P^j.
\]
\end{theorem}
The proof of Theorem~\ref{thm:limitrelax} is found in the appendix. 

{ 
An immediate consequence of 
Theorem~\ref{thm:limitrelax} is the following simple method for recovery of the orthogonal coordinate system from
the symmetric tensor, $\mathcal{Q}$ via the eigenframe of submatrices.

\begin{corollary}[$n$-Cross Recovery]
The $n$-cross can be recovered from a $\mathcal{Q} \in \Mcross$ by computing the eigenframe of any submatrix $\mathcal{Q}_{ij}$.
\end{corollary}
\begin{proof}
Since $\mathcal{Q}_{ij} = \sum_k P_{ij}^k P^k$ for $n\times n$ matrices $\{P^k\}_{k=1}^n$ which commute with each other. Therefore, by Lemma~\ref{lem:commute}  submatrices $\mathcal{Q}_{ij}$ have identical eigenframes for every $i,j \in \{1,\ldots,n\}$.
\end{proof}

\begin{rem}
 In the appendix, we also demonstrate that Theorem~\ref{thm:limitrelax} can be adapted to show that $\mathbb{M}^n_{\text{odeco}}$ can be identified by the constraint in \eqref{eq:odecodef}. 
\end{rem}
}

\subsection{Ginzburg-Landau relaxation off an $n$-cross field}

We will take elements of $\Mrelax$ and consider relaxations towards $\Mcross$ via two different Ginzburg-Landau approximations. As discussed in the introduction, we will relax our symmetric tensors $\mathcal{Q}$ by penalizing the potential 
\begin{equation}
\label{eq:WQ}
W(Q)=|\mathcal{Q}^2 - \mathcal{Q}|^2.
\end{equation}
In the bulk, this can be achieved by the energy, 
\begin{equation} \label{eq:earlyenergy}
\mathcal{E}(\mathcal{Q}) \equiv \frac{1}{2}\int_{\Omega} \LV \nabla \mathcal{Q} \RV^2 + {\frac{1}{\e^2}} \LV \mathcal{Q}^2 - \mathcal{Q} \RV^2 dx.
\end{equation}
By Theorem~\ref{thm:limitrelax} critical points of \eqref{eq:earlyenergy} will converge a.e. to a $n$-cross field as $\e \to 0$.  Following the discussion in Section~\ref{sec:boundary}, boundary singularities for $n$-cross fields are generically possible. To handle these scenarios we relax the condition on the boundary that the tensors are $n$-crosses; we handle this relaxation in two ways.
In the first case we impose boundary condition \eqref{eq:bc} as a hard constraint, and in the second case we  penalize our tensor for not aligning with the normal to the boundary.   

Let
\begin{align*}
    H^1_{\Mbold} & \equiv 
    \left\{ 
        \mathcal{A} \in H^1(\Omega; \Mrelax) \right\} \\
            H^1_{\Mbold, \nu} & \equiv 
    \left\{ 
        \mathcal{A} \in H^1(\Omega; \Mrelax) \hbox{ such that } \mathcal{A} \hbox{ satisfies }\eqref{eq:bc} \hbox{ on } \p \Omega \right\} 
\end{align*}
then the two Ginzburg-Landau relaxations are:

\noindent {\textbf{Method A}} Define the following energy
\begin{equation} \label{eq:GLenergy}
    \mathcal{E}(\mathcal{Q}) \equiv \frac{1}{2}\int_{\Omega} \LV \nabla \mathcal{Q} \RV^2 + \frac{1}{\e^2} \LV \mathcal{Q}^2 - \mathcal{Q} \RV^2 dx
    + \frac{1}{2 \delta_\e^2} \int_{\p \Omega} \LV \mathcal{Q}^2 - \mathcal{Q} \RV^2 ds,
\end{equation}
for $\mathcal{Q} \in H^1_{\Mbold,\nu}$.  In particular, we look for minimizers 
subject to the tensor constraints in Lemmas~\ref{lem:permute} and \ref{lem:traceQsub} 
and subject to boundary condition \eqref{eq:bc} with $\e\ll 1$ and $\delta_\e \ll 1$.  The nonlinear boundary relaxation allows for the formation of boundary vortices.  

\noindent {\textbf{Method B}} 
A weak anchoring version of the Ginzburg-Landau relaxation can be similarly defined, see for example \cite{Kleman, Newton} in the context of liquid crystals and \cite{BAUMAN2019447} in the context of Ginzburg-Landau theory. The corresponding weak anchoring version of our energy is 
\begin{equation} \label{eq:GLenergyWA}
    \mathcal{E}_{wa}(\mathcal{Q})
    \equiv 
    \frac{1}{2}\int_\Omega \LV \nabla \mathcal{Q} \RV^2
    + \frac{1}{\e^2}\LV \mathcal{Q}^2 - \mathcal{Q} \RV^2 dx 
    + \frac{1}{2 \delta_\e^2} \int_{\p \Omega} \LV \mathcal{Q}^2 - \mathcal{Q} \RV^2 ds
    + \frac{\lambda_\e}{2}\int_{\p \Omega}\left| \left[ \mathcal{Q}, \mathcal{P} \right] \right|^2 ds
\end{equation}
for $\mathcal{Q} \in H^1_{\Mbold}$ where $\mathcal{P} = P\otimes P$ with $P= \nu \otimes \nu$ for $\nu$ normal to the boundary. In this case, we take $\e \ll 1$, $\delta_\e \ll 1 $, and $\lambda_\e \gg 1$.

A natural numerical approach to generating an approximate $n$-cross field is to set up a constrained gradient descent of either \eqref{eq:GLenergy} with data in $H^1_{\mathbb{M}}$ or \eqref{eq:GLenergyWA} with data in $H^1(\Omega;\Mrelax)$ and choose $t \gg 1$. 

{All simulations in this paper are conducted using Method A. We expect that the results for Method B would produce similar outcomes as long as the parameter $\lambda_\e$ is sufficiently large; if $\lambda_\e$ is small, the effect of this would be analogous to what is known in the standard Ginzburg-Landau theory \cite{ABG}, where singular sets are observed to migrate from the interior of the domain to the boundary. The investigation of the shape of the resulting boundary singularities is beyond the scope of the present work.}
 
 \subsection{Removing constraints and the associated gradient descent}
 A simpler approach avoids dealing with the set of constraints that define our class of symmetric tensors.  
We first let $\mathbf{q}=(q_1,q_2,\ldots,q_k)^T$  be those components of $\mathcal{Q}$'s in $\Mrelax$ that are independent.  For example, $\mathcal{Q}_{1111} = q_1$, $\mathcal{Q}_{1112} = \mathcal{Q}_{1121} = q_2$, and so on.  Remark~\ref{rem:numberqs} shows that $k=2$ for $2$-frames and $k=9$ for $3$-frames.  
We can then redefine our Ginzburg-Landau energies as 
\begin{align} \label{eq:GLenergyq1}
    \mathcal{E}(\mathbf{q}) &\equiv \frac{1}{2}\int_{\Omega} \LV \nabla \mathcal{Q}(\mathbf{q}) \RV^2 + \frac{1}{\e^2} \LV \mathcal{Q}^2(\mathbf{q}) - \mathcal{Q}(\mathbf{q}) \RV^2 dx 
    + \frac{1}{2 \delta_\e^2} \int_{\p \Omega} \LV \mathcal{Q}^2(\mathbf{q}) - \mathcal{Q}(\mathbf{q}) \RV^2 ds \\
    \label{eq:GLenergyq2}
    \mathcal{E}_{wa}(\mathbf{q}) &\equiv \frac{1}{2}\int_{\Omega} \LV \nabla \mathcal{Q}(\mathbf{q}) \RV^2 + \frac{1}{\e^2} \LV \mathcal{Q}^2(\mathbf{q}) - \mathcal{Q}(\mathbf{q}) \RV^2 dx \\
    \nonumber
    & \qquad \qquad + \frac{1}{2 \delta_\e^2} \int_{\p \Omega} \LV \mathcal{Q}^2(\mathbf{q}) - \mathcal{Q}(\mathbf{q}) \RV^2 ds + \frac{\lambda_\e}{2}\int_{\p \Omega}\left| \left[ \mathcal{Q}(\mathbf{q}), \mathcal{P} \right] \right|^2 ds.
\end{align}
Our implementation follows the (unconstrained) gradient descent of \eqref{eq:GLenergyq1}, 
\begin{equation} \label{eq:gradientdescent}
    \p_t \mathbf{q} = - \nabla_{L^2} \mathcal{E}(\mathcal{Q}(\mathbf{q})), 
\end{equation}
subject to the appropriate natural boundary conditions.  

Since the focus of the current work is a practical algorithm for generating $n$-cross fields on Lipschitz domains, the analysis of these problems will be the subject of a follow-up paper.  




\section{$2$-cross fields}
\label{sec:2D}
We now apply our theory in two dimensions.  One particularly nice feature of the problem in this case is that the boundary conditions are Dirichlet conditions since the normal vector fully defines a $2$-cross field.  { After implementing the arguments from Section~\ref{sec:GLn} for $n=2$, we recover a Ginzburg-Landau relaxation for degree-${\frac14}$ vortices, as proposed and implemented in  earlier work, see \cite{BEAUFORT2017219,ViertelOsting}.}
\begin{lem} \label{lem:2DQdef}
Any $\mathcal{Q} \in \Mtworelax$ takes the form
\begin{equation} \label{e:2DQdef}
\mathcal{Q}= \begin{pmatrix}
\begin{pmatrix}
q_1 & q_2 \\ q_2 & 1- q_1 
\end{pmatrix}
& 
\begin{pmatrix}
q_2 & 1- q_1 \\ 1 - q_1 & - q_2 
\end{pmatrix} \\
\begin{pmatrix}
q_2 & 1- q_1 \\ 1 - q_1 & - q_2 
\end{pmatrix} & 
\begin{pmatrix}
1- q_1 & - q_2 \\ - q_2 & q_1 
\end{pmatrix}
\end{pmatrix}
\end{equation}
for $q_1,q_2 \in \R$.
\end{lem}

\begin{proof}
   We use Lemmas~\ref{lem:permute} and \ref{lem:traceQsub} to identify submatrix $\mathcal{Q}_{11}$. 
   These lemmas can be used for the other submatices.  For $\mathcal{Q}_{12} = \mathcal{Q}_{21}$, we use $\mathcal{Q}_{1211}  = \mathcal{Q}_{1112}$ and $\mathcal{Q}_{1212}  = \mathcal{Q}_{1122}$, along with 
	 $\tr{\mathcal{Q}_{12}}=0$, to identify the submatrix entries.  Finally, we use
	    $\mathcal{Q}_{2211} = \mathcal{Q}_{1122}$
	    $\mathcal{Q}_{2212}  = \mathcal{Q}_{1222}$
	and $\tr{\mathcal{Q}_{22}} = 1$ to identify the last submatrix.
	
\end{proof}

Let $\Omega$ be a domain in $\R^2$ with Lipschitz boundary.  For $\mathcal{Q} \in \Mtworelax$ we define the associated Ginzburg-Landau energy as
\begin{equation}
    \mathcal{E}_\e(\mathcal{Q}(\mathbf{q})) \equiv \frac{1}{2} \int_\Omega  \LV \nabla \mathcal{Q}(\mathbf{q}) \RV^2 
    + \frac{1}{\e^2} \LV \mathcal{Q}(\mathbf{q}) - \mathcal{Q}^2(\mathbf{q}) \RV^2 dx
\end{equation}
with  $\mq = (q_1,q_2)$.
Since 
\[
 \frac{1}{\e^2} \LV \mathcal{Q}(\mathbf{q}) - \mathcal{Q}^2(\mathbf{q}) \RV^2
 = \frac{16}{\e^2} \LC \LC q_1 - \frac{3}{4} \RC^2 + q_2^2 - \frac{1}{16} \RC^2,
\]
the  energy becomes
\begin{equation} \label{eq:2dGLenergyfinal}
    \mathcal{E}(\mathcal{Q}(\mathbf{q}))
    = 4 \int_\Omega \LV \nabla \mathbf{q} \RV^2 
    + \frac{8}{\e^2} \LC \LC q_1 - \frac{3}{4} \RC^2 + q_2^2 - \frac{1}{16} \RC^2 dx.
\end{equation}

Using \eqref{eq:GLenergyq1} and \eqref{eq:2dGLenergyfinal}, we arrive at the parabolic system:
\begin{align*}
    \p_t q_1 - \Delta q_1 & = \frac{16}{\e^2} 
\LC  q_1 - \frac{3}{4} \RC  \LC \LC q_1 - \frac{3}{4}\RC^2 + q_2^2 - \frac{1}{16} \RC \\
\p_t q_2 - \Delta q_2 & =\frac{16}{\e^2} q_2 \LC \LC q_1 - \frac{3}{4}\RC^2 + q_2^2 - \frac{1}{16} \RC.
\end{align*}
We supplement this parabolic system with the following Dirichlet boundary conditions,  
\begin{lem}
For $\mathcal{Q}$ satisfying \eqref{eq:bc} then \begin{equation} \label{eq:2Dbc}
    \begin{pmatrix}
    q_1 \\ q_2
    \end{pmatrix}
    = \begin{pmatrix}
    \nu_1^4 + \nu_2^4 \\ \nu_1^3 \nu_2 - \nu_1 \nu_2^3
    \end{pmatrix}
\end{equation}
on the boundary.  If $\nu = (\cos(\theta(x)), \sin(\theta(x)))^T$ then
\begin{equation}\label{eq:2Dbcangle}
    \begin{pmatrix}
    q_1 \\ q_2
    \end{pmatrix}
    =  \frac{1}{4} \begin{pmatrix}
     3 + \cos(4 \theta(x))  \\ \sin(4 \theta(x)) 
    \end{pmatrix}.
\end{equation}
\end{lem}
\begin{proof}
From $\mathcal{Q}_{11} \nu = \nu_{1}^2 \nu$ and $\nu_1^2 + \nu_2^2 = 1$ then
\begin{align*}
  \begin{pmatrix}
        \nu_1 & \nu_2 \\
        - \nu_2 & \nu_1
    \end{pmatrix}
    \begin{pmatrix}
        q_1 \\ q_2
    \end{pmatrix}
    =  \begin{pmatrix}
        \nu^3_1 \\ -\nu^3_2
    \end{pmatrix}.     
\end{align*}
\eqref{eq:2Dbc} follows.  Equation \eqref{eq:2Dbcangle} is a direct calculation.
\end{proof}
The form of \eqref{eq:2Dbcangle} points to the generic formation of degree-$\frac{1}{4}$ vortices in two dimensions.  This has been pointed out and studied { in \cite{BEAUFORT2017219, Knoeppel,ViertelOsting}.}  

\begin{rem}
If $\mba^1= (\cos(\theta), \sin(\theta))$ 
then a quick calculations show that the corresponding tensor satisfies
\[
e_\eta = 
\frac{1}{4} \left(
\begin{array}{cc}
 \left(
\begin{array}{cc}
 3+\cos (4 \theta) & \sin (4 \theta) \\
 \sin (4 \theta) & 1- \cos (4 \theta) \\
\end{array}
\right) & \left(
\begin{array}{cc}
 \sin (4 \theta) & 1-\cos (4 \theta) \\
 1- \cos (4 \theta) & -\sin (4 \theta) \\
\end{array}
\right) \\
 \left(
\begin{array}{cc}
 \sin (4 \theta) & 1- \cos (4 \theta) \\
 1- \cos (4 \theta) & -\sin (4 \theta) \\
\end{array}
\right) & \left(
\begin{array}{cc}
 1- \cos (4 \theta) & - \sin (4 \theta) \\
 -\sin (4 \theta) & 3 + \cos (4 \theta) \\
\end{array}
\right) \\
\end{array}
\right).
\]
Indeed the tensor satisfies symmetries in \eqref{e:2DQdef}. Furthermore, the $4\theta$ in each argument implies a fundamental domain of $[0,{\pi/2})$ which corresponds to the symmetry group structure.
\end{rem}


\section{$3$-cross fields}
\label{sec:3D}
We now turn to our primary objective - a practical algorithm for generating $3$-crosses in Lipschitz domains.  As in two dimensions, we first identify the higher-order $\mathcal{Q}$ tensor.

\begin{lem} \label{lem:q3d}
Any $\mathcal{Q} \in \Mthreerelax$ takes the form
\begin{align*}
    \mathcal{Q} = 
    \begin{pmatrix}
        \mathcal{Q}_{11} & \mathcal{Q}_{12} & \mathcal{Q}_{13}\\
        \mathcal{Q}_{12} & \mathcal{Q}_{22} & \mathcal{Q}_{23}\\
        \mathcal{Q}_{13} & \mathcal{Q}_{23} & \mathcal{Q}_{33}
    \end{pmatrix}
    \end{align*}
where
\begin{align*}
    \mathcal{Q}_{11} & = 
    \begin{pmatrix}
        q_1 & q_2 & q_3 \\
        q_2 & q_4 & q_5 \\
        q_3 & q_5 & 1 - q_1 - q_4
    \end{pmatrix} \\
    \mathcal{Q}_{12} & = 
    \begin{pmatrix}
        q_2 & q_4 & q_5 \\
        q_4 & q_6 & q_7 \\
        q_5 & q_7 &  - q_2 - q_6
    \end{pmatrix} \\
    \mathcal{Q}_{13} & = 
    \begin{pmatrix}
        q_3 & q_5 & 1 - q_1 - q_4 \\
        q_5 & q_7 & - q_2 - q_6 \\
        1 - q_1 - q_4 & - q_2 - q_6 &  - q_3 - q_7
    \end{pmatrix} \\
    \mathcal{Q}_{22} & = 
    \begin{pmatrix}
        q_4 & q_6 & q_7 \\
        q_6 & q_8 & q_9 \\
        q_7 & q_9 & 1 - q_4 - q_8
    \end{pmatrix} \\
    \mathcal{Q}_{23} & = 
    \begin{pmatrix}
        q_5 & q_7 & -q_2-q_6 \\
        q_7 & q_9 & 1-q_4-q_8 \\
        -q_2-q_6 & 1-q_4-q_8 & - q_5 - q_9
    \end{pmatrix} \\
    \mathcal{Q}_{33} & = 
    \begin{pmatrix}
        1-q_1-q_4 & -q_2-q_6 & -q_3-q_7 \\
        -q_2-q_6 & 1-q_4-q_8 & - q_5 - q_9 \\
        -q_3-q_7 & - q_5 - q_9 & q_1 + 2q_4+q_8-1
    \end{pmatrix} 
\end{align*}
\end{lem}
\begin{proof}
We  use Lemmas~\ref{lem:traceQsub} and \ref{lem:permute} to identify all submatrices. $\mathcal{Q}_{11}$ follows from symmetry and the trace-one condition.  Next for $\mathcal{Q}_{12}$ we use $\mathcal{Q}_{1212} = \mathcal{Q}_{1122}$ and $\mathcal{Q}_{1213} = \mathcal{Q}_{1123}$, 
along with the trace-free condition.
For $\mathcal{Q}_{13}$ we use $\mathcal{Q}_{1311} = \mathcal{Q}_{1113}$, $\mathcal{Q}_{1312} = \mathcal{Q}_{1123}$, $\mathcal{Q}_{1322} = \mathcal{Q}_{1223}$, and $\mathcal{Q}_{1323} = \mathcal{Q}_{1233}$ with
 the trace-free condition.
For $\mathcal{Q}_{22}$ we use $\mathcal{Q}_{2211} = \mathcal{Q}_{1122}$, $\mathcal{Q}_{2212} = \mathcal{Q}_{1222}$, $\mathcal{Q}_{2213} = \mathcal{Q}_{1223}$ and the trace-one condition.  
For $\mathcal{Q}_{23}$ we use $\mathcal{Q}_{2311} = \mathcal{Q}_{1123}$, $\mathcal{Q}_{2312} = \mathcal{Q}_{1123}$, $\mathcal{Q}_{2313} = \mathcal{Q}_{1233}$, $\mathcal{Q}_{2322} = \mathcal{Q}_{2223}$, $\mathcal{Q}_{2323} = \mathcal{Q}_{2233}$, 
along with the trace free condition.
Finally, for $\mathcal{Q}_{33}$ we use $\mathcal{Q}_{3311} = \mathcal{Q}_{1133}$, $\mathcal{Q}_{3312} = \mathcal{Q}_{1233}$, $\mathcal{Q}_{3313} = \mathcal{Q}_{1333}$, $\mathcal{Q}_{3322} = \mathcal{Q}_{2233}$, $\mathcal{Q}_{3323} = \mathcal{Q}_{2333}$ and 
the trace-one condition.

\end{proof}

{ We note that the expression in Lemma~\ref{lem:q3d} is the same as the one for the fourth order tensor $\mathbb{A}_{ijkl}$ in \cite{Schaft}.}

We now consider relaxations of $\mathcal{Q}$ by assuming that $\mathbf{q}:=(q_1,\ldots,q_9)\in\mathbb R^9$ is arbitrary and imposing the penalty
\begin{equation}
    \label{e:Wq3D}
    W(\mathbf{q})\equiv\left|\mathcal{Q}(\mathbf{q})^2-\mathcal{Q}(\mathbf{q})\right|^2.
\end{equation}
Following \eqref{eq:GLenergyq1}, the associated Ginzburg-Landau energy becomes
\begin{equation}
    \label{eq:efin}
    \mathcal{E}(\mathcal{Q}(\mathbf{q}))
    = \frac{1}{2} \int_\Omega \left[\LV \nabla \mathcal{Q}(\mathbf{q}) \RV^2 + \frac{1}{\e^2} W(\mathbf{q})\right] dx
    + \frac{1}{2\delta_\e^2} \int_{\p \Omega}  W(\mathbf{q}) ds
\end{equation}
with $W(\mathbf{q})$ defined above in \eqref{e:Wq3D}.

We now generate the boundary conditions in three dimensions. 

\begin{lem}
\label{lem:bound}
Let $\nu = (\nu_1,\nu_2,\nu_3)^T$ be an outward normal to the boundary.  
For $\mathcal{Q} = \mathcal{Q}(\mathbf{q})$ satisfying \eqref{eq:bc} then
 $\mathbf{q}$ satisfies the following set of constraints on the boundary
\begin{equation}
\label{eq:bcsys}
\left(
\begin{array}{ccccccccc}
\nu_1 & \nu_2 & \nu_3 & 0 & 0 & 0 & 0 & 0 & 0 \\
0 & \nu_1 & 0 & \nu_2 & \nu_3 & 0 & 0 & 0 & 0 \\
-\nu_3 & 0 & \nu_1 & -\nu_3 & \nu_2 & 0 & 0 & 0 & 0 \\
0 & 0 & 0 & \nu_1 & 0 & \nu_2 & \nu_3 & 0 & 0 \\ 
0 & -\nu_3 & 0 & 0 & \nu_1 & -\nu_3 & \nu_2 & 0 & 0 \\
0 & 0 & 0 & 0 & 0 & \nu_1 & 0 & \nu_2 & \nu_3 \\
0 & 0 & 0 & -\nu_3 & 0 & 0 & \nu_1 & -\nu_3 & \nu_2
\end{array}
\right)
\left(
\begin{array}{c}
q_1 \\ q_2 \\ q_3 \\ q_4 \\ q_5 \\ q_6 \\ q_7 \\ q_8 \\ q_9
\end{array}
\right)=
\left(
\begin{array}{c}
\nu_1^3 \\ \nu_1^2\nu_2 \\ -\left(\nu_2^2+\nu_3^2\right)\nu_3 \\ \nu_1\nu_2^2 \\ \nu_1\nu_2\nu_3 \\ \nu_2^3 \\ -\left(\nu_1^2+\nu_3^2\right)\nu_3
\end{array}
\right).
\end{equation}
The matrix on the left has rank 7. 
\end{lem}

\begin{proof}
Equation~\eqref{eq:bcsys} follows from 
$\mathcal{Q}_{11} \nu = \nu_{1}^2 \nu$.  
The matrix rank follows by a direct calculation.  
\end{proof}

Note that Lemma \ref{lem:bound} prescribes only seven conditions on the nine variables $q_i,\ i=1,\ldots,9.$ The remaining two conditions are then the natural boundary conditions for the variational problem associated with the energy \eqref{eq:efin}.

\begin{rem}
We note that if $\nu = (1,0,0)$, then the boundary condition reduces to two dimensions and the boundary conditions in three dimensions.  In particular, assume that $(1,0,0)$ is an eigenvector of every $3\times3-$block $\mathcal{Q}_{ij}$ with the eigenvalue $\delta_{1i}\delta_{1,j}$
\begin{equation} \label{eq:3Dreduce2D}
\hspace{-23mm}\left(\hspace{-.3cm}\begin{array}{ccc} \left(\begin{array}{ccc} 1 & 0 & 0 \\ 0 & 0 & 0 \\ 0 & 0 & 0 \end{array}\right) & \left(\begin{array}{ccc} 0 & 0 & 0 \\ 0 & 0 & 0 \\ 0 & 0 & 0 \end{array}\right) & \left(\begin{array}{ccc} 0 & 0 & 0 \\ 0 & 0 & 0 \\ 0 & 0 & 0 \end{array}\right) \\\\ \left(\begin{array}{ccc} 0 & 0 & 0 \\ 0 & 0 & 0 \\ 0 & 0 & 0 \end{array}\right) & \left(\begin{array}{ccc} 0 & 0 & 0 \\ 0 & q_8 & q_9 \\ 0 & q_9 & 1-q_8\end{array}\right) & \left(\begin{array}{ccc} 0 & 0 & 0\\ 0 & q_9 & 1-q_8 \\ 0 & 1-q_8 & -q_9 \end{array}\right) \\\\ \left(\begin{array}{ccc} 0 & 0 & 0 \\ 0 & 0 & 0 \\ 0 & 0 & 0 \end{array}\right) & \hspace{-.4cm} \left(\begin{array}{ccc} 0 & 0 & 0\\ 0 & q_9 & 1-q_8 \\ 0 & 1-q_8 & -q_9 \end{array}\right)  & \hspace{-.4cm} \left(\begin{array}{ccc} 0 & 0 & 0 \\ 0 & 1-q_8 & -q_9 \\ 0 & -q_9 & q_8\end{array}\right) \end{array}\hspace{-.3cm}\right)
\end{equation}
We note the similarity between \eqref{e:2DQdef} and \eqref{eq:3Dreduce2D}.
\end{rem}

\section{Numerical Examples in 3D}
\label{sec:examples}
In this section we use the finite elements software package COMSOL \cite{comsol} to find solutions of the Euler-Lagrange equations for the functional \eqref{eq:efin}, subject to the constraints \eqref{eq:bcsys} on the boundary. In what follows, we refer to this equation as the Ginzburg-Landau PDE. For each domain geometry we ran a gradient flow simulation {starting from a constant initial condition until the numerical solution reached an equilibrium}. The system of PDEs that we solve is given in the Appendix \ref{ap:2}. {The parameters $\varepsilon$ and $\delta_\varepsilon$ were taken to be small, typically around $10\%$ of the domain size}. Note that there is a relationship between $\varepsilon$ and $\delta_\varepsilon$ that determines whether the topological defects of minmimizers of \eqref{eq:efin} lie on the boundary or the interior of the domain \cite{ABG}. {As we already stated above, we do not investigate this issue further in the present paper}. 

\subsection{Cube with a cylindrical notch}
The first simulation was run for a domain in the shape of a cube with a cylindrical notch (Figs.~\ref{fig0}-\ref{fig01}) and was motivated by an example in \cite{viertel2016analysis}. A critical solution of the Ginzburg-Landau PDE recovered via gradient flow shows that the vertical line remains one of directions of the $3$-cross everywhere in the domain. The solution has one disclination line depicted in blue in the left inset in Fig.~\ref{fig0}. The $3$-cross distribution in a horizontal cross-section of the domain at the level that includes the notch is shown in Fig.~\ref{fig01}. The trace of the disclination in this cross-section is circled in red. The right inset in Fig.~\ref{fig0} shows three families of streamlines along the lines of the $3$-cross field. 

\subsection{Spherical shell}
Here we solve the Ginzburg-Landau PDE in a three-dimensional shell that lies between the spheres of radii $0.495$ and $0.5$ with $\varepsilon=\delta_\varepsilon=0.02$, subject to the system of constraints \eqref{eq:bcsys} on both boundaries of the shell. As expected, the solution gives the array of eight vortices shown in Fig.~\ref{fig5}. These vortices are actually short disclination lines that connect the components of $\partial\Omega$. The distribution of $3$-crosses on one eighth of the outer sphere is shown in Fig.~\ref{fig5.5}. The vortex of degree $1/4$ is indicated by the red ellipse.

\begin{figure}
\begin{center}
\includegraphics[scale=.8]{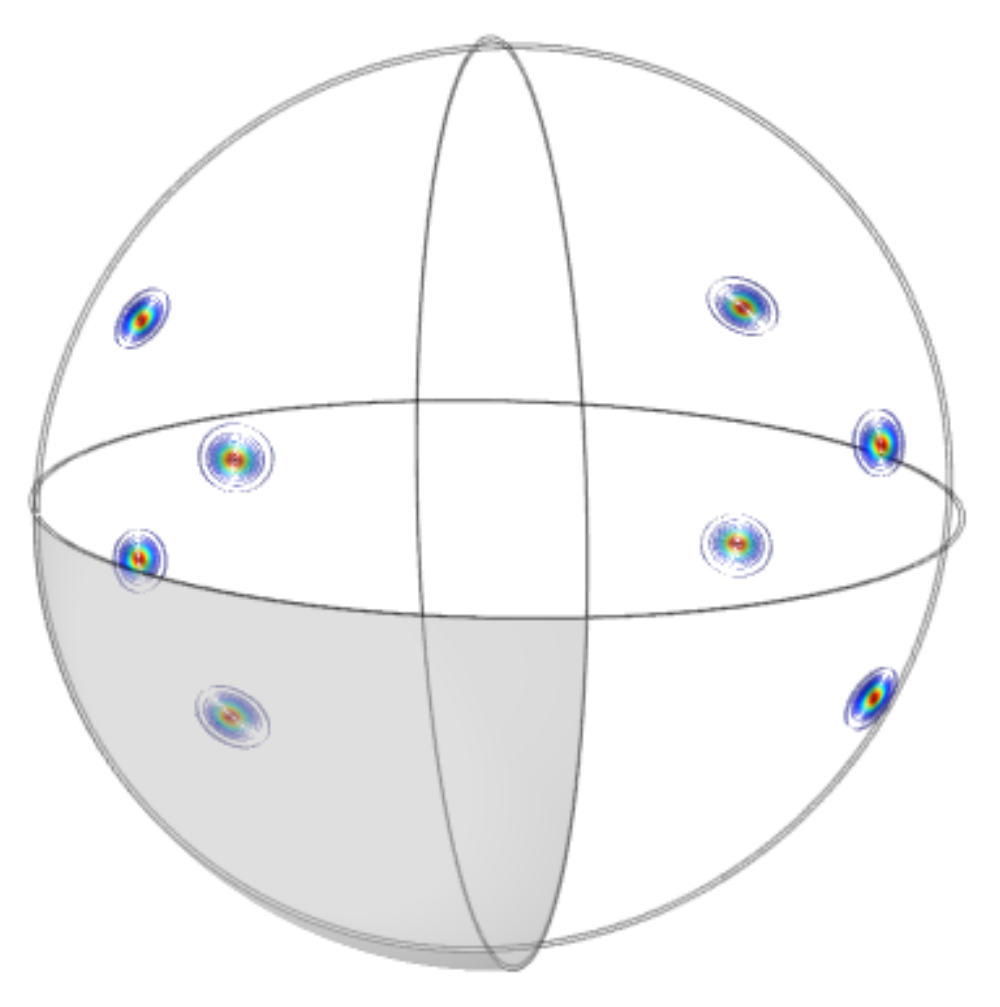}
\caption{Vortices on a thin spherical shell. The figure shows the contour plot of the potential $W(Q)$.}\label{fig5}
\end{center}
\end{figure}
\begin{figure}
\begin{center}
\includegraphics[scale=.5]{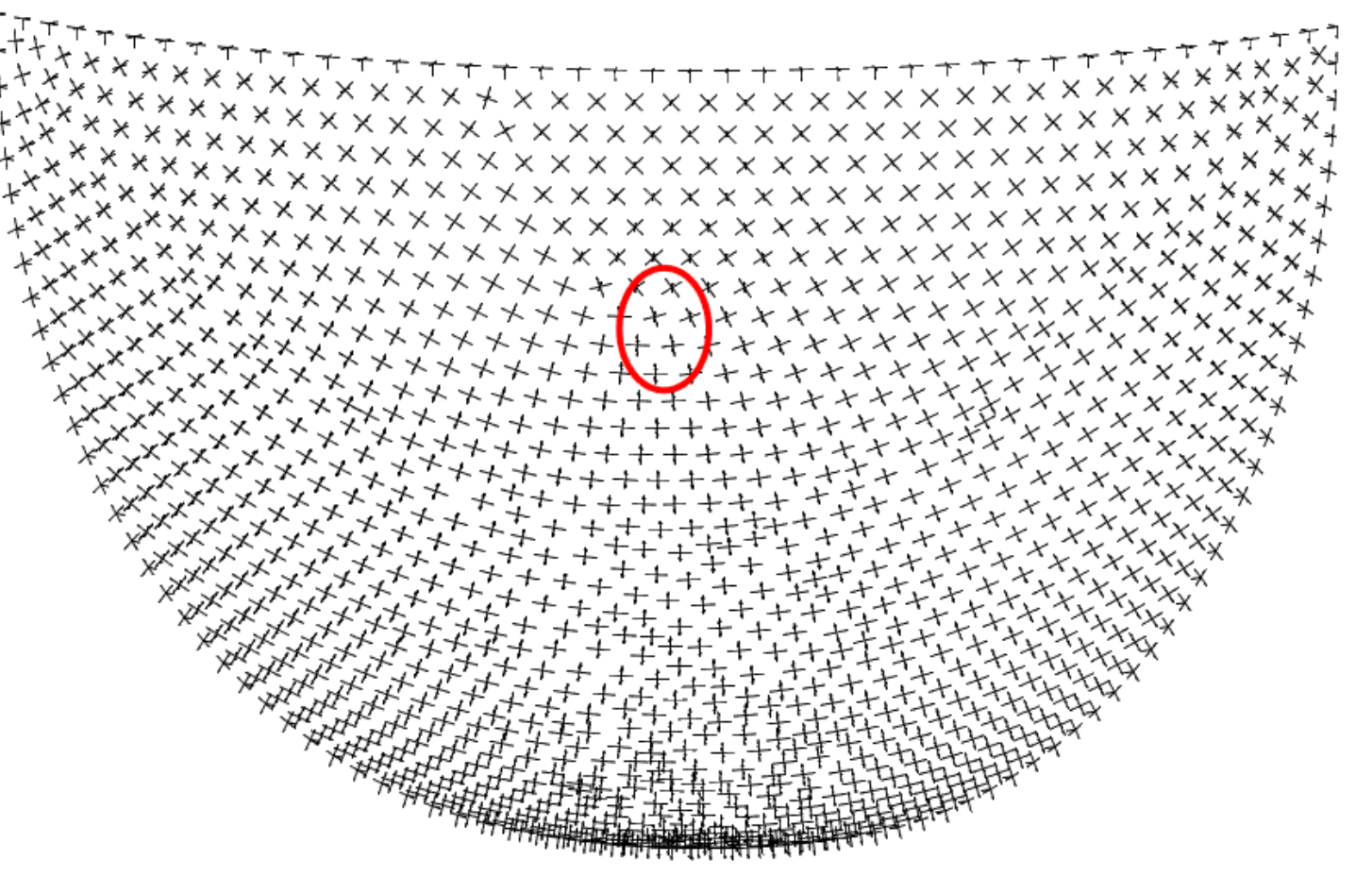}
\caption{The $3$-cross distribution on the shaded one eighth of the spherical shell in Fig.~\ref{fig5}. Note that each octant in Fig.~\ref{fig5} has exactly one vortex associated with it.}\label{fig5.5}
\end{center}
\end{figure}

\subsection{Ball}
Simulations in a ball resulted in Figs. \ref{fig6}-\ref{fig7}. One can see a similar pattern of surface vortices as in the case of a spherical shell, now connected by the line singularities that run close to the surface of the ball. The cross-section of the $3$-cross field in the ball are depicted in Fig.~\ref{fig7}. The lengths of the frame vectors inside the disclination cores are scaled to make the intersections between the disclinations and $xy$-plane more visible. 

\begin{figure}
\begin{center}
\includegraphics[height=3.5in,width=3.9in]{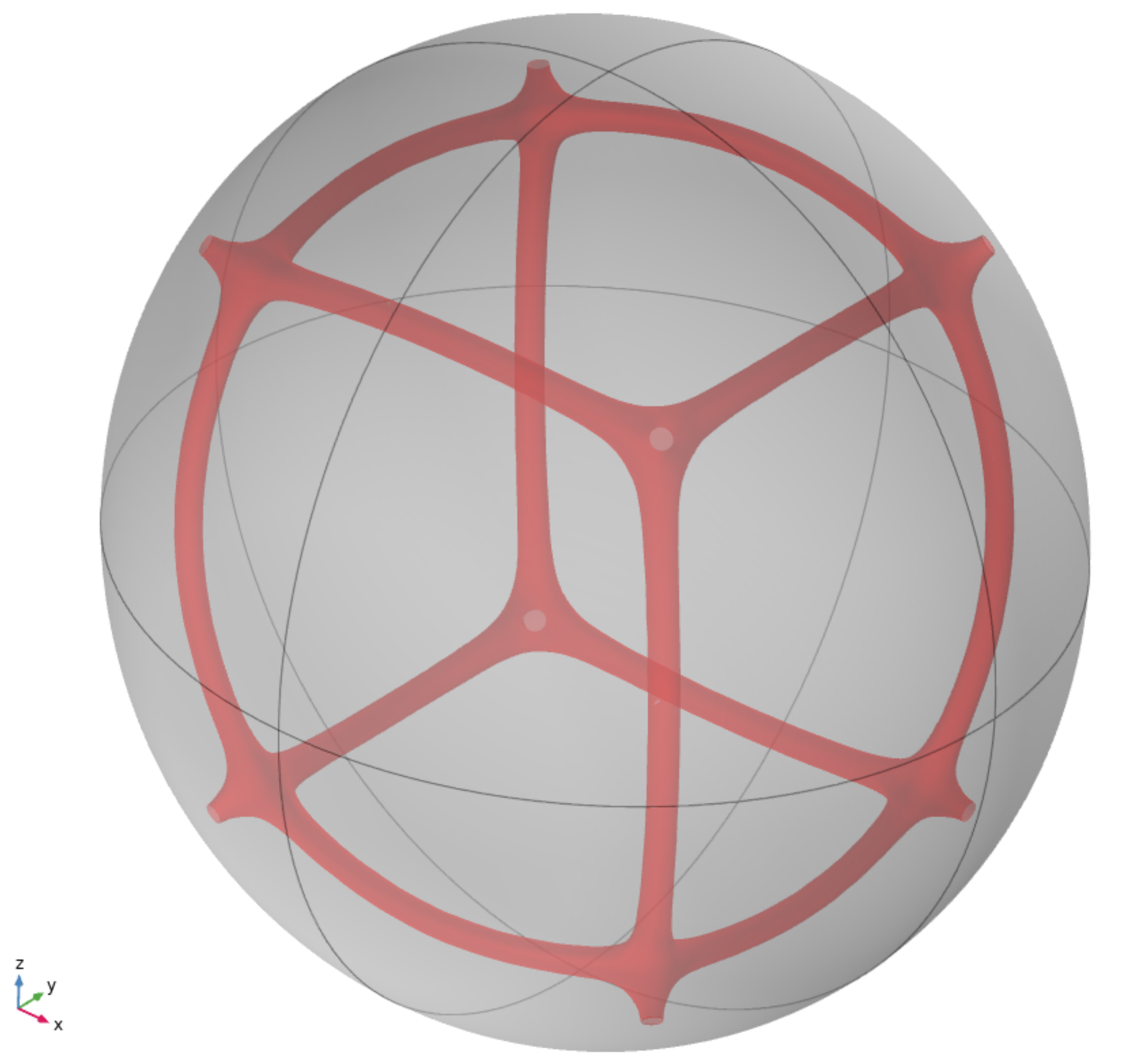}
\caption{Singularities in a ball. The contour plot of the potential $W(Q)$ indicates that there are eight surface vortices connected by eight disclination lines.}\label{fig6}
\end{center}
\end{figure}
\begin{figure}
\begin{center}
\includegraphics[scale=.7]{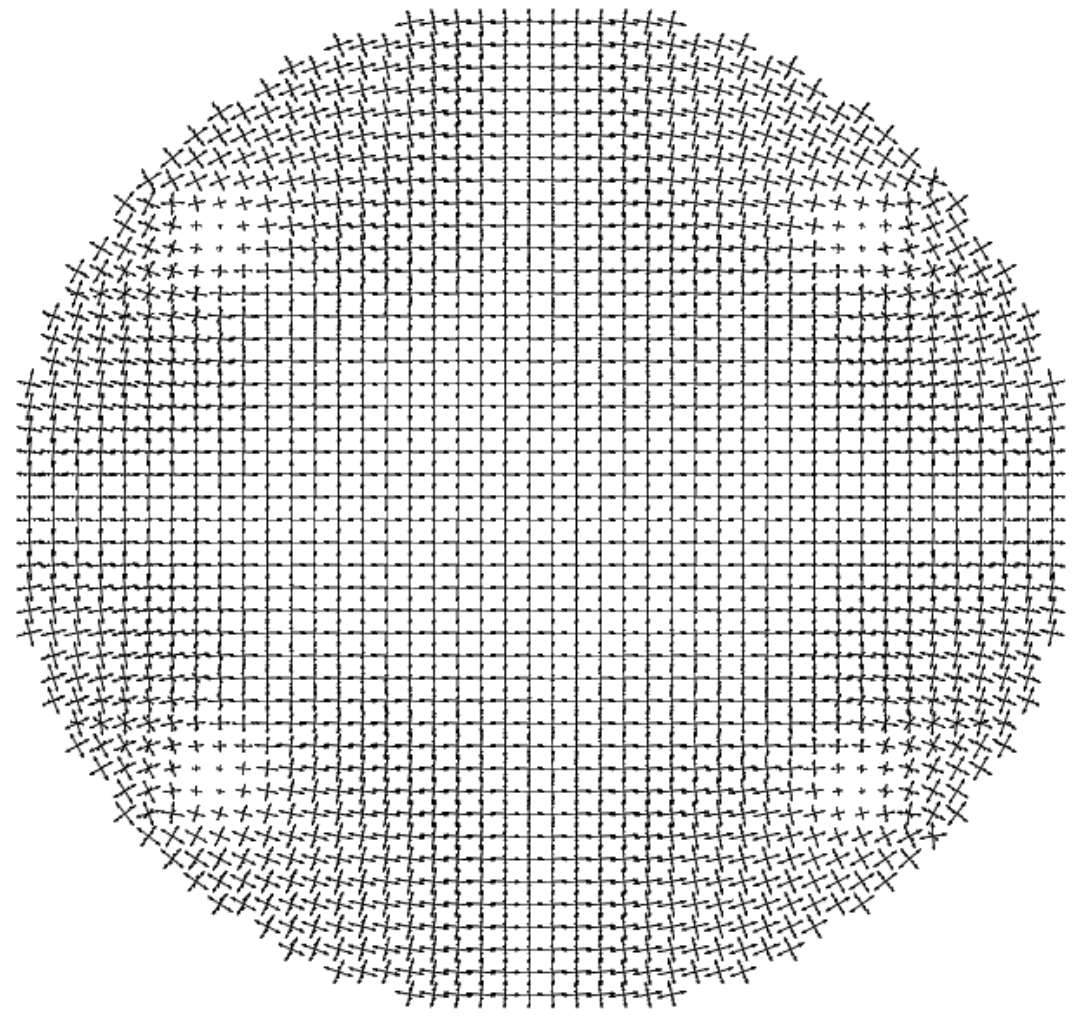}
\caption{Cross-section along the $xy$-plane of the $3$-cross field inside a ball.}\label{fig7}
\end{center}
\end{figure}

\subsection{Toroidal domain with a cylindrical hole}
The next example deals with the domain in a shape of a toroid with a cylindrical hole (Fig.~\ref{fig8}, left), motivated by an example in \cite{viertel2016analysis}. One can see in the right inset in Fig.~\ref{fig8} that four line singularities are present in the undrilled part of the torus. This is expected since two out of the free line fileds that form a $3$-cross should have the winding number $1$ along the circumference of the torus and the $3$-cross makes four turns along the same path. This suggests that there are four line singularities in this part of the domain as it should be energetically preferable for a degree one singularity to split into four degree $1/4$ singularities of the same type. The cross-sections of the $3$-cross field in the sphere are depicted in Fig.~\ref{fig9}.

\begin{figure}
\begin{center}
\includegraphics[scale=.4]{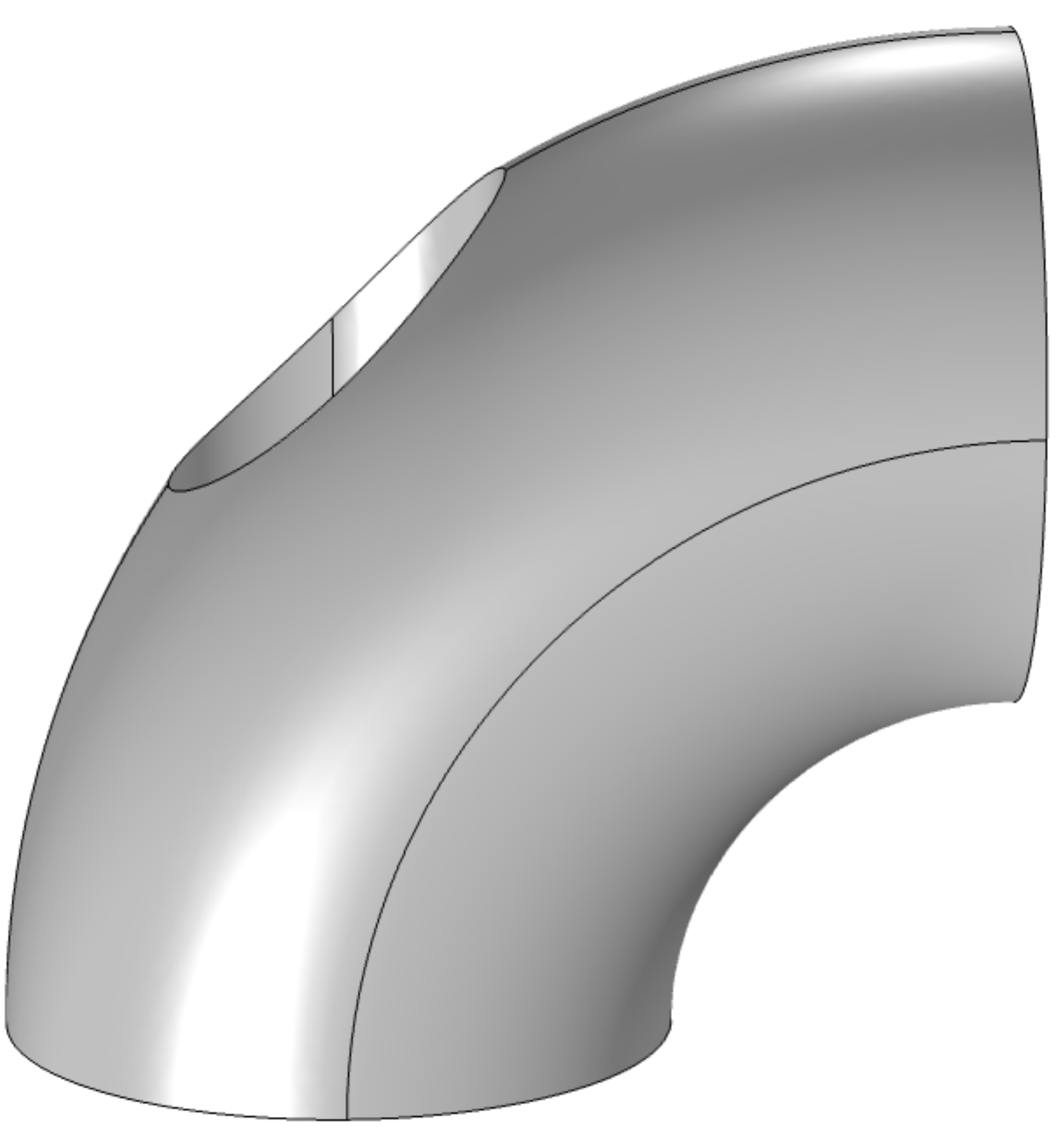} \quad \includegraphics[scale=.4]{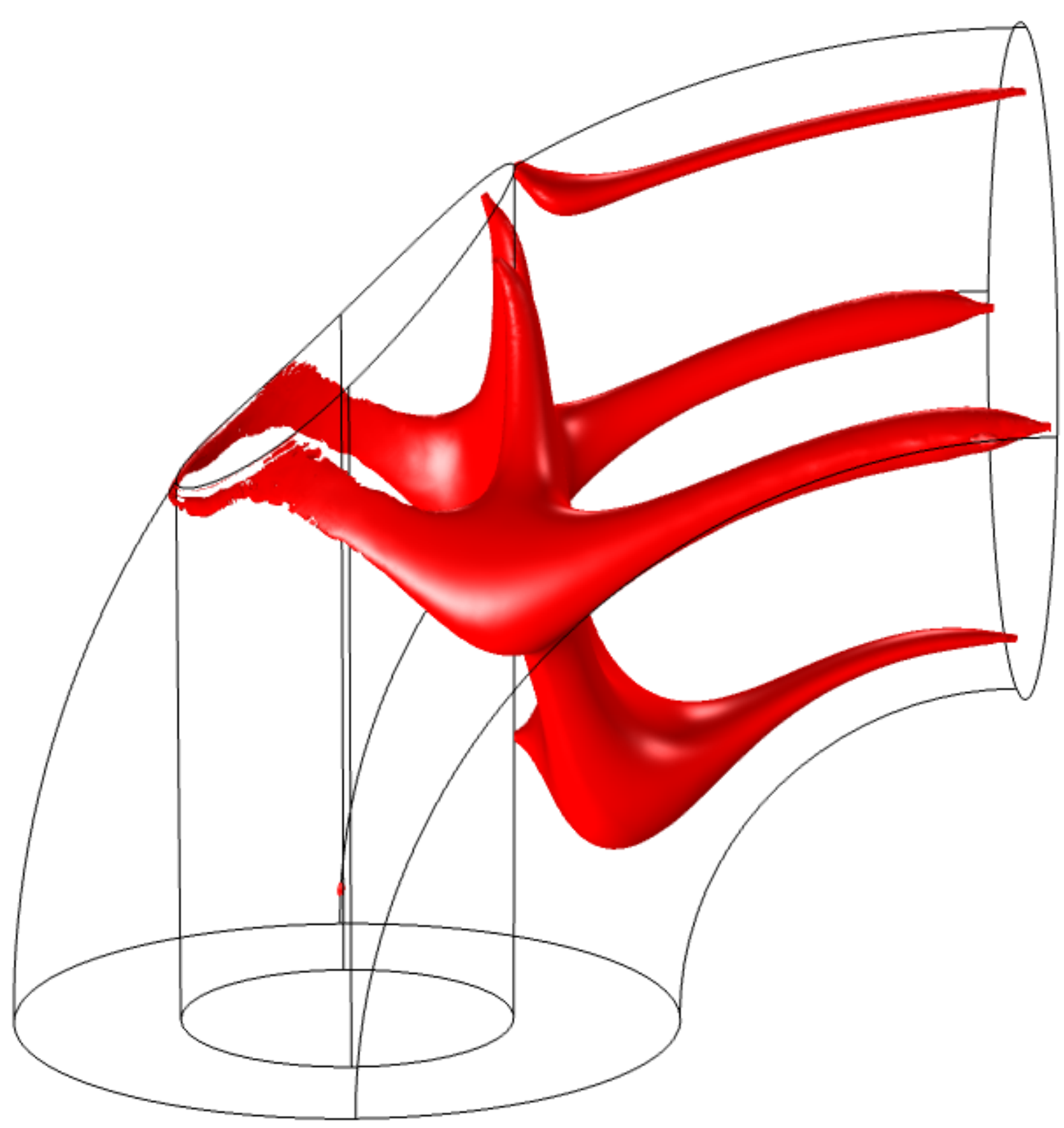}
\caption{Disclinations on the toroidal domain with a cylindrical hole. The contour plot of the potential $W(Q)$ indicates that there are four disclination lines in the part of the domain away from the hole.}\label{fig8}
\end{center}
\end{figure}
\begin{figure}
\begin{center}
\includegraphics[scale=.7]{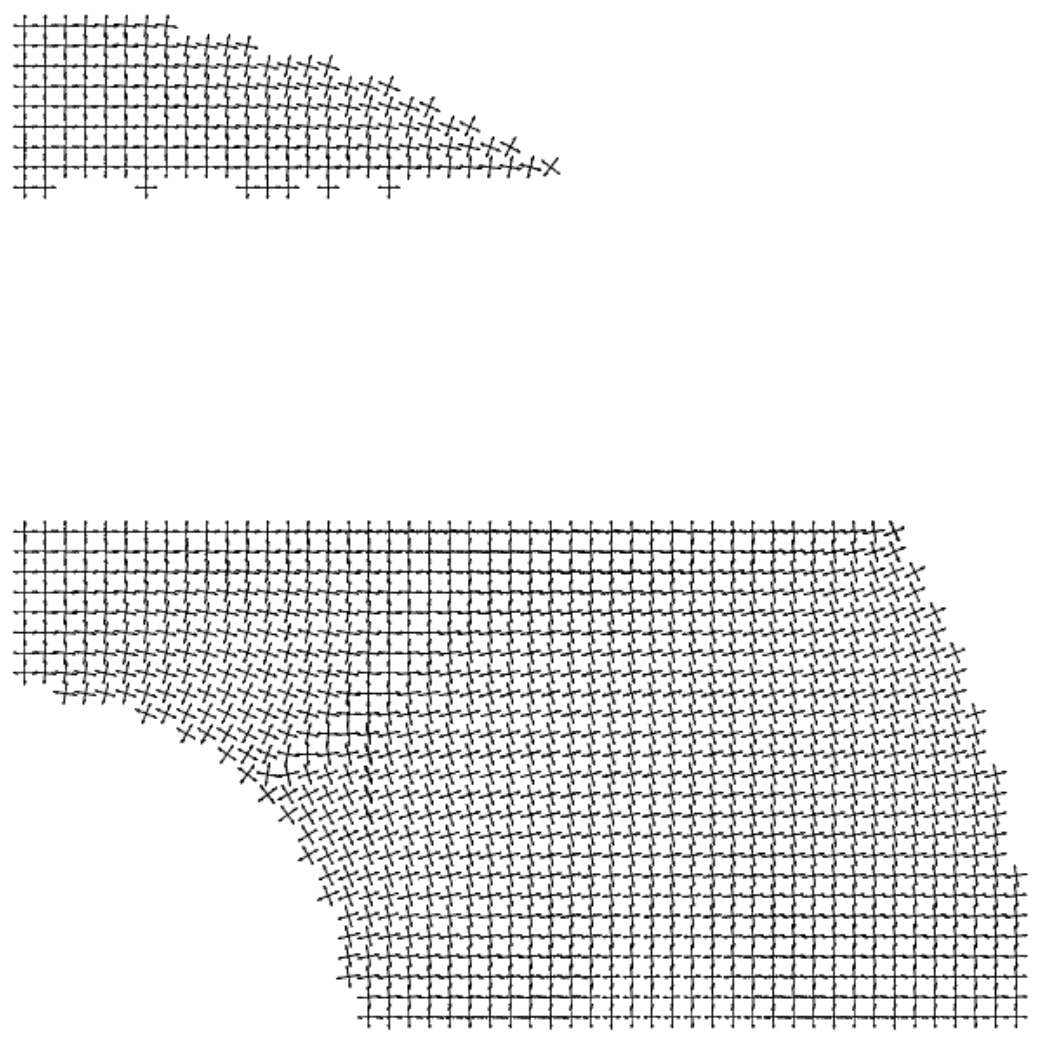} \qquad \includegraphics[scale=.7]{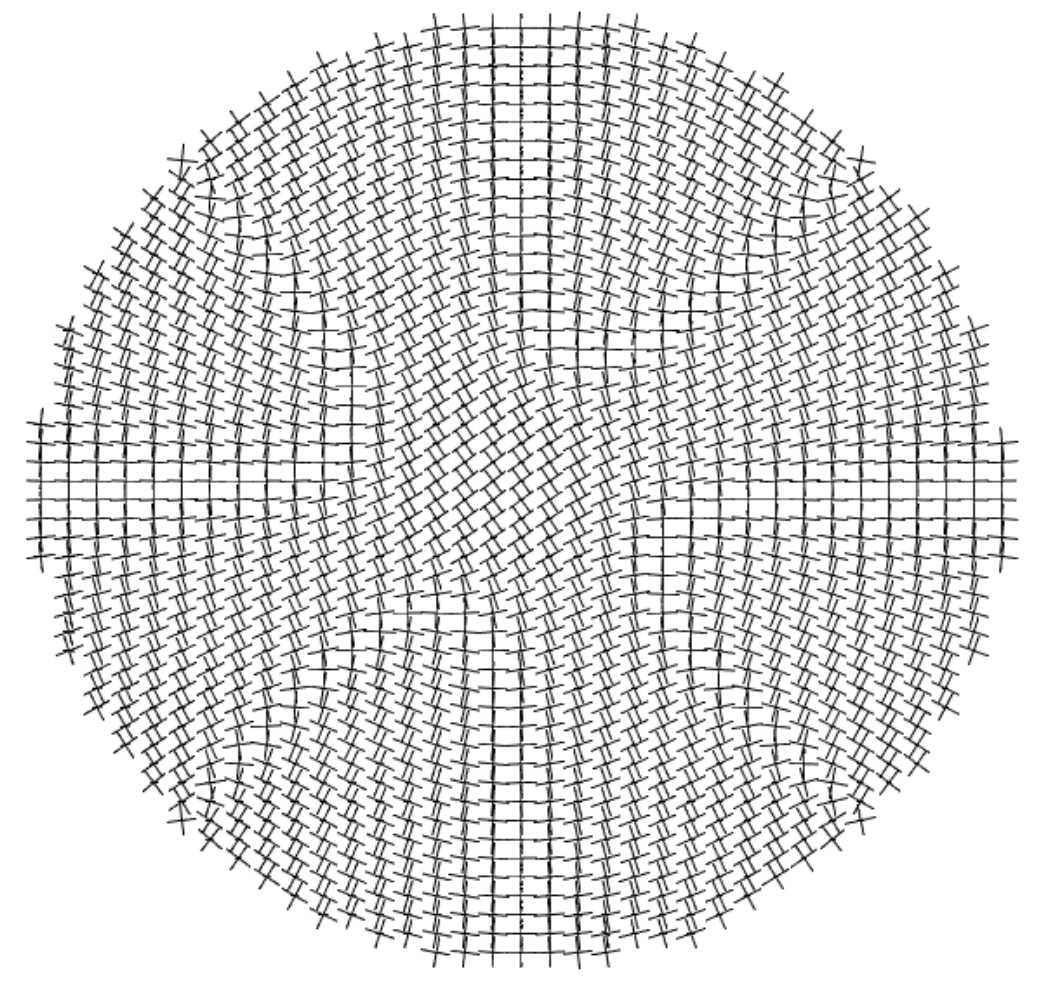} \\ \vspace{5mm} \includegraphics[scale=.7]{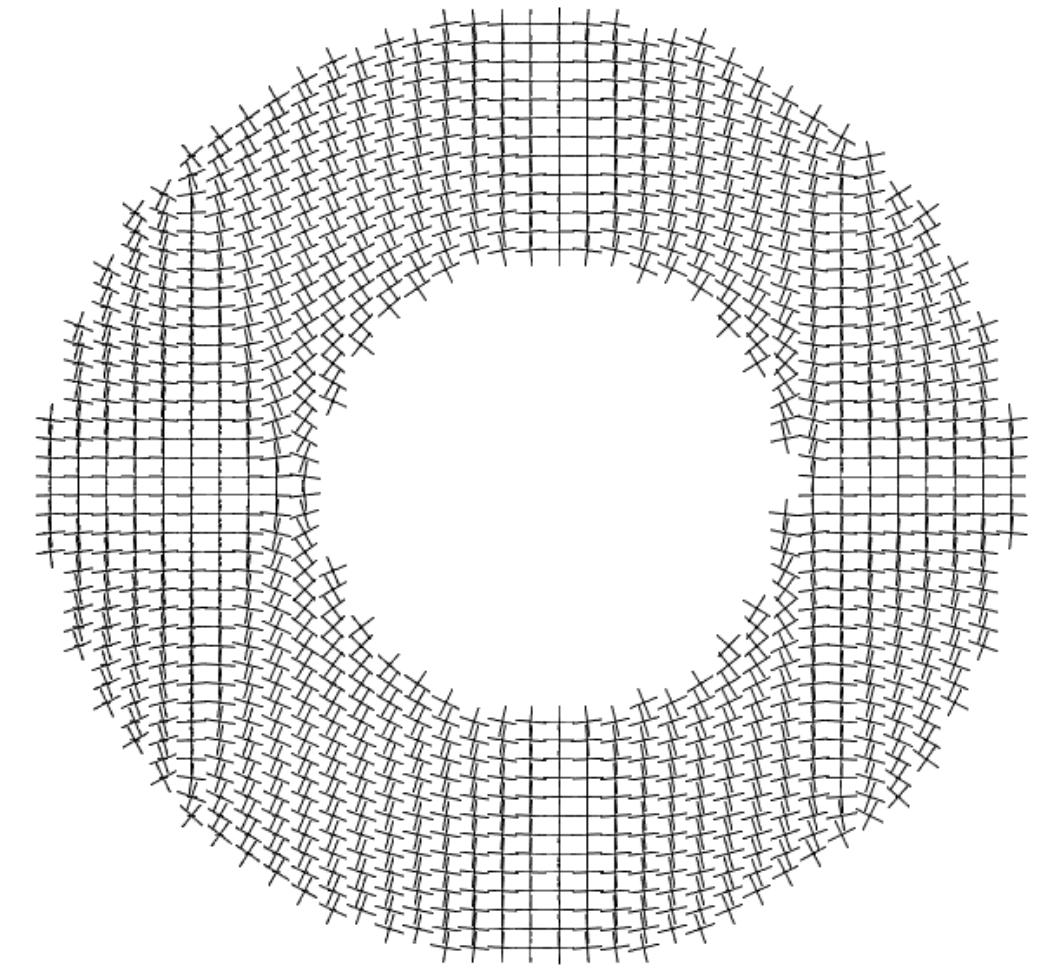}
\caption{Cross-sections of the $3$-cross field inside the toroidal domain with a cylindrical hole along $xy$-, $xz$-, and $yz$-planes, respectively.}\label{fig9}
\end{center}
\end{figure}

\subsection{Domains with complex geometries}
The last set of examples (Fig.~\ref{fig10}) shows disclinations networks in domains with complex geometries. The common features dictated by the topology of a domain include, for example,  disclinations associated with rounded corners, four disclination lines associated with a cylindrical hole in a rectanguar cylinder, and the absence of disclinations in a cylindrical region with a concentric cylindrical hole. 
\begin{figure}
\begin{center}
\includegraphics[scale=.35]{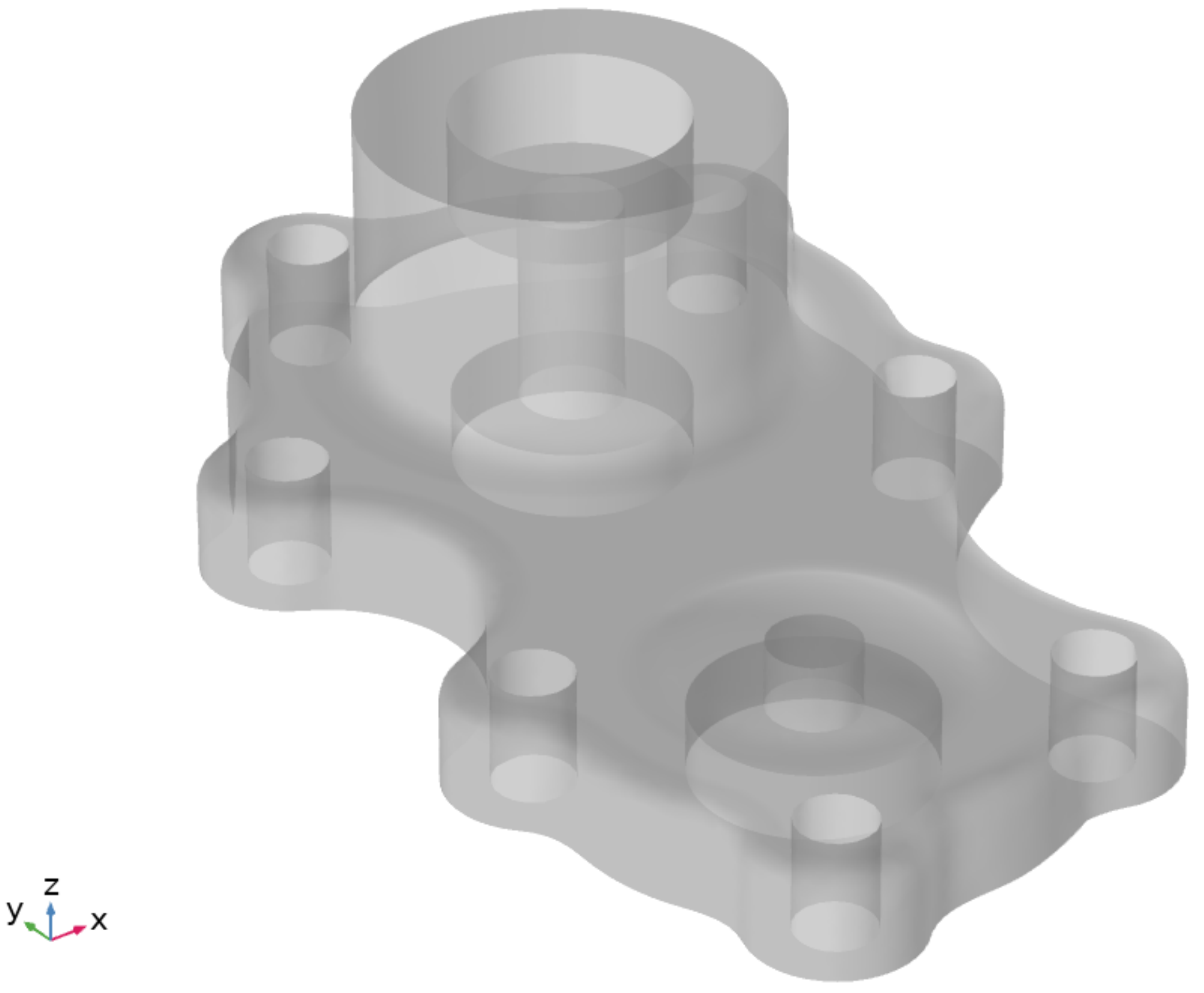} \qquad \includegraphics[scale=.35]{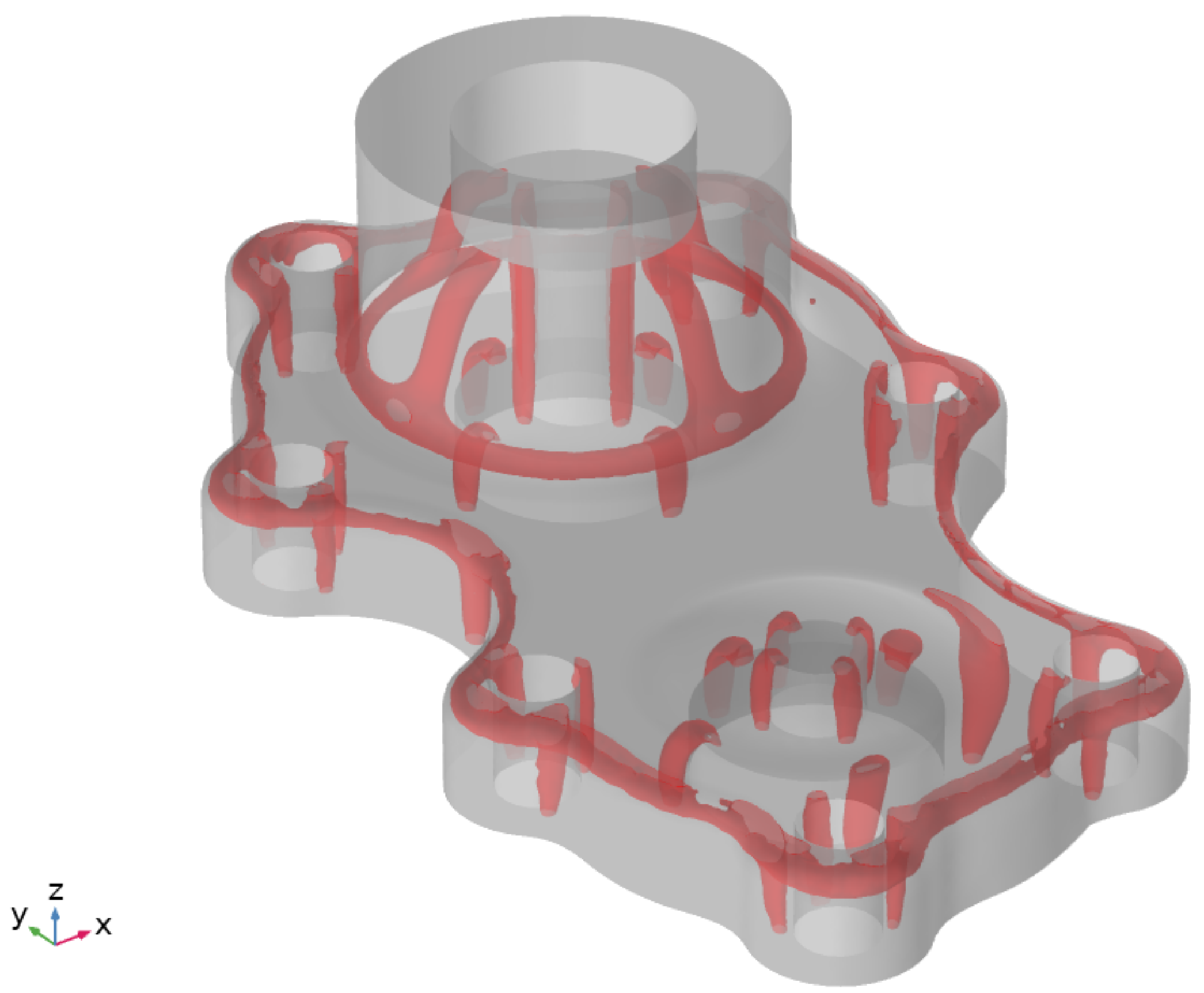} \\ \vspace{5mm} \includegraphics[scale=.4]{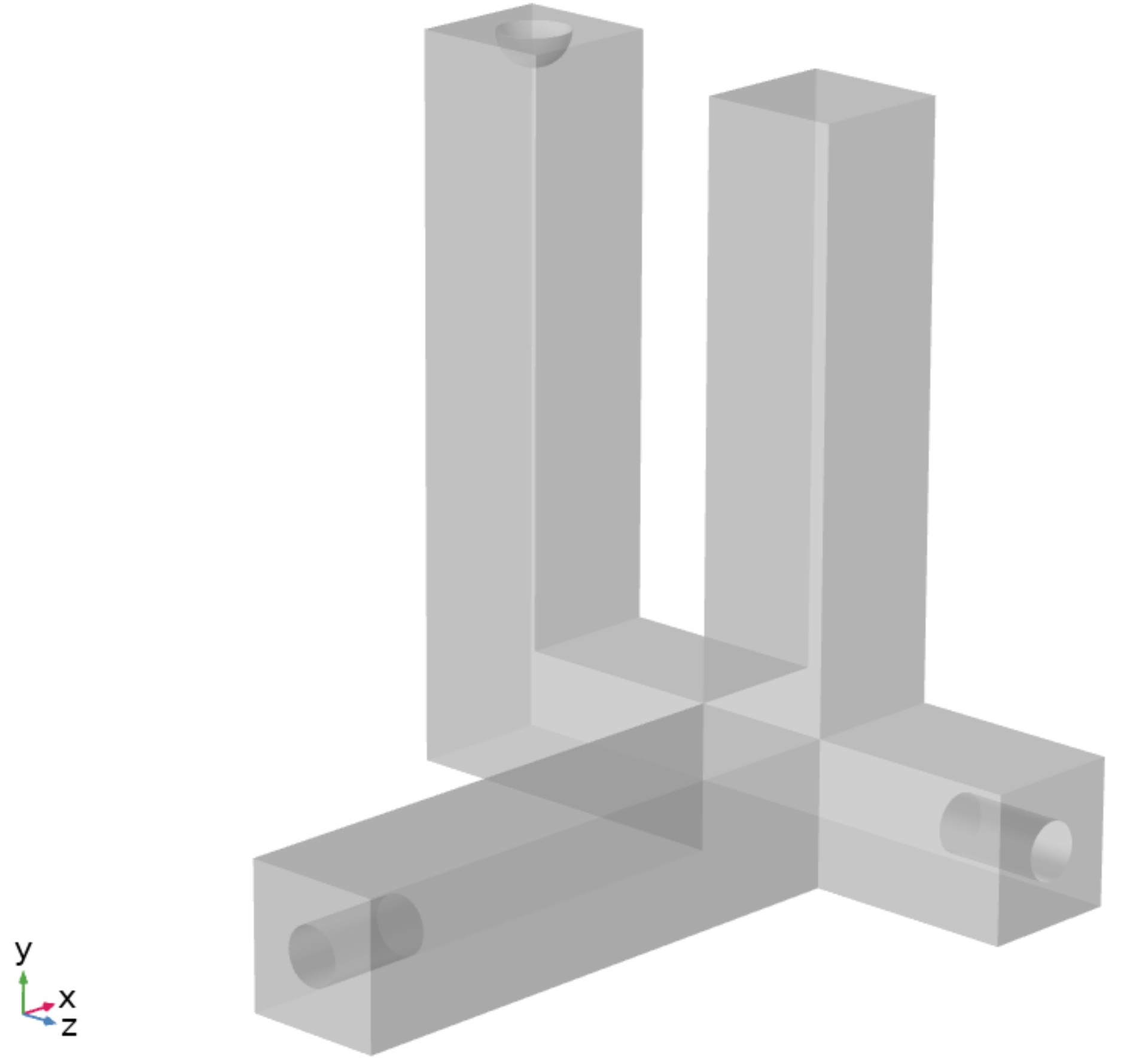} \qquad \includegraphics[scale=.4]{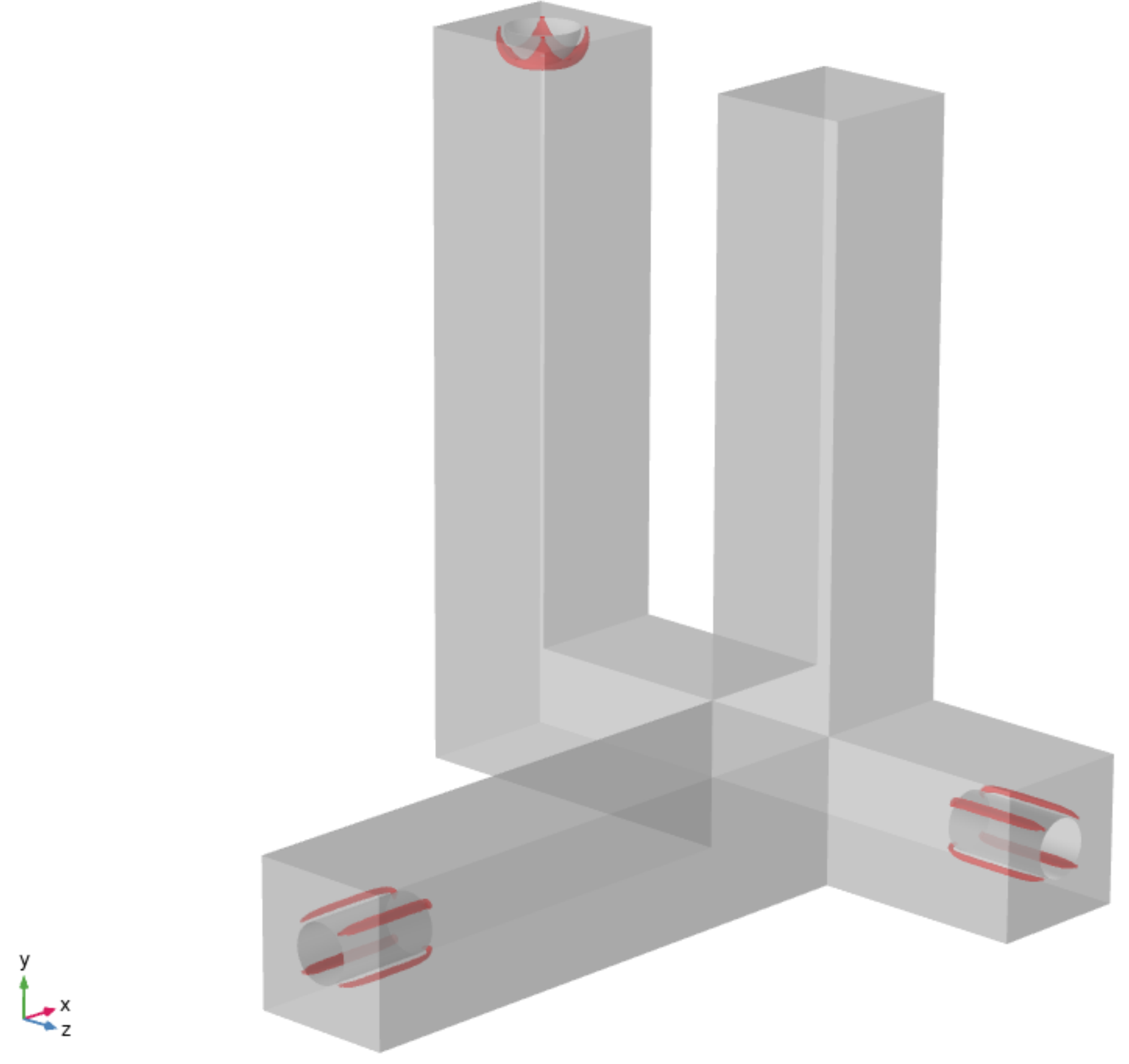}
\caption{Disclinations in the domains with complex geometries. The right column shows the contour plot of the potential $W(Q)$ associated the domain shown in the left column.}\label{fig10}
\end{center}
\end{figure}

{
\section{Discussion}
\label{sec:discussion}
Both 2- and 3-cross fields are finding increasing use in a diverse set of applications, including in scientific computing and computer graphics.  As  methods are developed, it is of rising importance to establish an analytic framework that explains the limiting behavior of these numerically generated cross fields. 

A rigorous treatment of $n$-cross fields initiated in this paper provides a useful framework to better understand common features and challenges in the numerical generation of cross fields that conform to the boundary of Lipschitz domains. The work presented here provides a novel relaxation by identifying a natural Ginzburg-Landau energy where the fourth order symmetric tensor is smooth, but remains close to the associated $n$-cross field.  In particular one can avoid generic singular sets, as have been seen numerically and analytically. This new relaxation also provides a new way to visualize the development of singular sets, and this is done by looking for the set where the fourth order tensor $\mathcal{Q}$ fails to be a projection.  Furthermore, we identify a natural boundary relaxation. The associated Ginzburg-Landau energy leads to a new and potentially interesting Calculus of Variations problem which can be used to study the behavior of singular sets.  

There are a number of interesting directions of  inquiry.  One very important question is to better understand and characterize the behavior of the limiting singular set.  This singular set, where the fourth order symmetric tensor $\mathcal{Q}$ fails to be a projection, seems to concentrate on a co-dimensions two rectifiable set.  Identifying and establishing the associated $\Gamma$-limit of the energy \eqref{eq:GLenergyq1} may provide some insight into the structure of the singular set, including the seeming generic development of quadratic junctions in the singular set, see Figure~\ref{fig6}.

Since we are far from understanding the global minimizing behavior of the limiting cross fields, another avenue of study is to see whether other relaxation methods, such as the MBO-based method for odeco tensors found in \cite{palmer2019algebraic} converges to the same limiting cross field.
Indeed one can relax towards the odeco variety instead of $n$-cross fields via the Ginzburg-Landau energy:
\begin{equation} \label{eq:ODECOrelaxation}
E_{Od,\e} = {\frac12} \int_{\Omega} |\nabla \mathcal{Q}|^2
+ \frac{1}{\varepsilon^2} | ( \mathbb{I} - \mathbb{P}_{Sym} ) (\mathcal{Q}^2) |^2 + \frac{1}{\e^2} | ( \mathbb{I} - \mathbb{P}_{Sym} ) (\mathcal{Q}^4) |^2 dx
\end{equation}
for $\mathcal{Q} \in \Mrelax$,
where $\mathbb{P}_{Sym}$ is defined in \eqref{eq:projsymtensor}.  Understanding the difference in the singular sets of these two Ginzburg-Landau relaxations is another avenue of study. Here establishing an analog of the Hairy Ball Theorem for odeco varieties is an interesting question.

Finally, it is of interest to see if there is an explicit way to generate the nearest $n$-cross field from any fourth order symmetric tensor with trace conditions. Such an explicit representation could allow for a much faster MBO-based method for generating an $n$-cross field.  
}

\section{Appendix A:  Proof of Theorem \ref{thm:limitrelax}}
\label{sec:8}


In this appendix we prove Theorem \ref{thm:limitrelax}.  To do this we will need to set up our notation, and we do that first.  After that we provide the proof of the Theorem.

\subsection{Notation}

In this section ${\mathcal L}(F)$ denotes the set of linear maps from the vector space $F$ to itself.  


The symbol $\Mall$ will denote the set of all $n\times n$ matrices with real entries.  We will continue to write
$$
\langle A, B \rangle = {\rm tr}(B^TA)
$$
for $A, B\in \Mall$.  The spaces $\Mallnsqr$ and ${\mathcal L}(\Mall)$ are isomorphic and we consider an explicit isomorphism identifying $\Mall$ with $\R^{n^2}$. To this end, if $B\in \Mall$, we write $B=({\bf b}^1 | ... | {\bf b}^n)$ where ${\bf b}^j\in \R^n$, $j = 1, ..., n$, is the $j$-th column of $B$.  Then we set
\begin{align}
& \bX: \Mall \to \R^{n^2} \nonumber\\
&B \to \bX(B) = \bX_B = 
\begin{pmatrix}
{\bf b}^1 \\ 
\vdots \\
{\bf b}^n
\end{pmatrix}\label{e:define_bold_X}
\end{align}
In other words, we identify $B$ with the vector in $\R^{n^2}$ that has the columns of $B$ stacked up vertically.  We will use the notations ${\bf X}(B)$ and ${\bf X}_B$ interchangeably.  It is easy to check that we have
$$
\bX_A\cdot \bX_B = \langle A , B \rangle = {\tr}(B^TA)
$$
for all $A, B\in \Mall$, where $\bX_A\cdot \bX_B$ denotes the standard dot product in $\mathbb{R}^{n^2}$.  

Note that with this identification between $\Mall$ and $\mathbb{R}^{n^2}$ we can consider $A, B \in \Mall$ and define the rank-$1$, $n^2\times n^2$ matrix
$$
\bX_{A}\bX_{B}^T.
$$
Matrices of this form will appear repeatedly in the rest of this section.

With this particular identification we define
\begin{equation}\label{def_second_id}
\Phi_0 : \Mallnsqr \to {\mathcal L}(\Mall)
\end{equation}
by the conditions that (a) $\Phi_0$ be linear and (b) for any $A, B, C \in \Mall$
$$
\Phi_0(\bX_A \bX_B^T)(C) = \langle B, C \rangle A.
$$
Well-known properties of tensor products show that this defines $\Phi_0$ completely.

The condition that defines $\Phi_0$ can be equivalently stated as follows: if $A, B, C \in \Mall$, and $R_{A, B}\in {\mathcal L}(\Mall)$ is defined by
\begin{equation}\label{eq:def_R_A_B}
R_{A, B}(C) = \langle B, C\rangle A,
\end{equation}
then
\begin{equation}\label{eq:rel_mat_lin_map}
R_{A, B} = \Phi_0({\bf X}_A {\bf X}_B^T).
\end{equation}
For $A=B$ we will write $R_{A}$ instead of $R_{A, A}$.  

\medskip
\medskip

Yet a third way to interpret the definitions of $\Phi_0$ and ${\bf X}$ is the following:  for every ${\mathcal Q}\in \Mallnsqr$ and every $A\in \Mall$, if ${\mathcal Q}_L = \Phi_0({\mathcal Q})$, then
\begin{equation}\label{eq:rel_Q_lin_map}
{\bf X}_{{\mathcal Q}_L(A)} = {\mathcal Q}{\bf X}_A.
\end{equation}
Here we interpret ${\mathcal Q}{\bf X}_A$ as the $n^2\times n^2$ matrix ${\mathcal Q}$ multiplying the vector ${\bf X}_A \in\R^{n^2}$ in a standard fashion, whereas on the left hand side ${\mathcal Q}_L(A)$ denotes the linear map ${\mathcal Q}_L$ from $\Mall$ to itself, acting on the matrix $A\in \Mall$.  This is the standard identification between $\Mallnsqr$ and ${\mathcal L}(\Mall) \cong {\mathcal L}(\R^{n^2})$ that comes from the identification $\Mall \cong \R^{n^2}$ provided by the isomorphism ${\bf X} : \Mall \to \R^{n^2}$.  In particular, this shows that if $A\in \Mall$, then ${\bf X}_A$ is an eigenvector of $\mathcal Q$, if and only if $A$ is an eigenvector of ${\mathcal Q}_L = \Phi_0({\mathcal Q})$ with the same eigenvalue.

For later reference it will be useful to have concrete expressions for the $n^2\times n^2$ matrices of two elements of ${\mathcal L}(\Mall)$.  We record them here.  The first one is the matrix of $R_{A, B}$ defined in \ref{eq:def_R_A_B}.  Note that \ref{eq:rel_mat_lin_map} already gives us an expression for the matrix of $R_{A, B}$.  More precisely, the matrix ${\mathcal Q}^R_{A, B}\in \Mallnsqr$ defined by the equation
$$
{\bf X}_{R_{A, B}(C)} = {\mathcal Q}^R_{A, B} \,\,  {\bf X}_C
$$
can be expressed as
$$
{\mathcal Q}^R_{A, B} = {\bf X}_A{\bf X}_B^T.
$$

The second map from ${\mathcal L}(\Mall)$ we will refer to later is $L_{A, B}\in {\mathcal L}(\Mall)$, for $A, B\in \Mall$, defined by the equation
\begin{equation}\label{eq:matrix_pre_post_mult}
L_{A, B}(C) = ACB^T
\end{equation}
for all $C\in \Mall$.  We write $L_A$ when $A=B$.  A direct computation shows that the matrix ${\mathcal Q}^L_{A, B} \in \Mallnsqr$ defined by the equation
$$
{\bf X}_{L_{A, B}(C)} = {\mathcal Q}^L_{A, B} \,\,  {\bf X}_C
$$
for all $C\in \Mall$ can be expressed as
$$
{\mathcal Q}^L_{A, B} = \left (   \begin{array}{ccc}  B_{11}A & \cdots & B_{1n}A  \\ \vdots & \ddots & \vdots \\ B_{n1}A  & \cdots & B_{nn}A \end{array} \right ).
$$

\subsection{Permutation Operators}

Recall $S_4$, the group of permutation of the set $\{1, 2, 3, 4\}$.  For $\sigma\in S_4$, define
\begin{equation}\label{def:permutation_op}
T_\sigma : \Mallnsqr\to \Mallnsqr
\end{equation}
by the conditions that $T_\sigma$ be linear and
$$
T_\sigma({\bf X}_{{\bf u}^1 ({\bf u}^2)^T}^T {\bf X}_{{\bf u}^3 ({\bf u}^4)^T}^T) = {\bf X}_{{\bf u}^{\sigma(1)} ({\bf u}^{\sigma(2)})^T} {\bf X}_{{\bf u}^{\sigma(3)} ({\bf u}^{\sigma(4)})^T}^T
$$
for every ${\bf u}^1, {\bf u}^2, {\bf u}^3, {\bf u}^4\in \R^n$.  Note that, for ${\bf a}, {\bf b}\in \R^{n}$, ${\bf a}{\bf b}^T \in \Mall$.  Again, standard facts about tensor products show that this condition defines $T_\sigma$ completely.


\medskip
\medskip

For later reference we record  expressions of $T_\sigma$ for the following three permutations:
\begin{align}
\sigma_1(1, 2, 3, 4) &= (3, 4, 1, 2), \nonumber \\
\sigma_2(1, 2, 3, 4) &= (2, 1, 3, 4) \,\,\,\,\mbox{and} \nonumber \\
\sigma_3(1, 2, 3, 4) &= (1, 3, 2, 4),\label{distinguished_permutations}
\end{align}
and write $T_j$ instead instead of $T_{\sigma_j}$.  Direct computations starting from ${\bf X}_{{\bf u}^1 ({\bf u}^2)^T}{\bf X}_{{\bf u}^3({\bf u}^4)^T}^T$ give the
\begin{proposition}
We have the identities
\begin{equation}\label{eq:transp_full_mat}
T_1({\mathcal Z}) = {\mathcal Z}^T
\end{equation}
for all ${\mathcal Z}\in \Mbold^{n^2}(\R{})$, as well as
\begin{equation}
T_1(\bX_A \bX_B^T) = \bX_B \bX_A^T,
\end{equation}
\begin{equation}\label{eq:trans_first_mat}
T_2(\bX_A \bX_B^T) = \bX_{A^T} \bX_B^T
\end{equation}
and
\begin{equation}\label{eq:cols_to_blocks}
T_3(\bX_A \bX_B^T) = {\mathcal Q}^L_{A, B}
\end{equation}
for all $A, B \in \Mall$, where ${\mathcal Q}^L_{A, B}$ is the matrix of the linear map $L_{A, B}$ defined in \ref{eq:matrix_pre_post_mult}.  
\end{proposition}
\begin{rem}
Let $P^1, ..., P^n\in \Mall$ satisfy $(P^j)^2 = (P^j)^T = P^j$, $P^jP^k = P^kP^j = \delta_{kj}P^j$, ${\rm tr}(P^j)=1$, and
$$
\sum_{j=1}^n P^j = I_n.
$$
Define ${\mathcal Q}\in \Mallnsqr$ by
$$
{\mathcal Q} = \sum_{j=1}^n {\bf X}_{P^j} {\bf X}_{P^j}^T.
$$
It is easy to check that $T_3({\mathcal Q}) = {\mathcal Q}$.  This plus \ref{eq:cols_to_blocks} shows that the equation above provides an equivalent definition for ${\mathcal Q}$ defined in \ref{e:Qdef}.
\end{rem}





We now turn to the proof of Theorem \ref{thm:limitrelax}.  As we shall see, the proof below gives both Theorem \ref{thm:limitrelax}, as well as the fact that the odeco variety, defined in \ref{eq:def_Robeva}, is equal to that defined by equation \ref{eq:odecodef}.
\begin{proof}[Proof of Theorem~\ref{thm:limitrelax}]
For ${\mathcal Q} \in \Mallnsqr$ define
$$
{\mathcal Q}_L = \Phi_0({\mathcal Q}) \in {\mathcal L}(\Mall),
$$
where $\Phi_0$ is the isomorphism defined in \ref{def_second_id}.  Rather than assuming immediately that
$$
T_\sigma({\mathcal Q}) = {\mathcal Q}^2 = {\mathcal Q}
$$
for all $\sigma \in S_4$, we assume ${\mathcal Q}\in \Mallnsqr$ satisfies ${\mathcal Q}\neq 0$ and
$$
T_\sigma({\mathcal Q}) = {\mathcal Q}, \,\,\, T_\sigma({\mathcal Q}^2) = {\mathcal Q}^2 \,\,\,\mbox{and}\,\,\, T_\sigma({\mathcal Q}^4) = {\mathcal Q}^4
$$
for all $\sigma \in S_4$.  This is equivalent to assuming that
$$
\mathbb{P}_{Sym}({\mathcal Q}) = {\mathcal Q}, \,\,\,\mathbb{P}_{Sym}({\mathcal Q}^2) = {\mathcal Q}^2, \,\,\,\mbox{and}\,\,\,\mathbb{P}_{Sym}({\mathcal Q}^4) = {\mathcal Q}^4,
$$
where $\mathbb{P}_{Sym}$ is the projection operator defined in \ref{eq:projsymtensor}.  The reason for assuming this is to make this proof work for both Theorem~\ref{thm:limitrelax} and Corollary~\ref{cor:ODECO}.

The proof consists of three main steps: First, use standard linear algebra to write ${\mathcal Q}$ as a linear combination of projections, and show further that the images of these projections contain only symmetric matrices.  Second, we show that the $n\times n$ matrices in the images of these projections commute with each other.  Third, we use this to finish the proof.  

For the first step we proceed as follows.  First recall that $T_\sigma({\mathcal Q}) = {\mathcal Q}$ for all $\sigma \in S_4$.  From equation \ref{eq:transp_full_mat} we conclude that
$$
T_1({\mathcal Q}) = {\mathcal Q}^T = {\mathcal Q}.
$$
By the standard spectral theorem we deduce the existence of an orthonormal basis of the image of ${\mathcal Q}$, denoted by $Q^1, ..., Q^m \in \Mall$, and real numbers $\lambda_1, ..., \lambda_m \in \mathbb{R}\setminus \{0\}$, such that
$$
{\mathcal Q}_L(A) = \sum_{j=1}^m \lambda_j \langle A, Q^j\rangle Q^j = \sum_{j=1}^n \lambda_j R_{Q^j}(A),
$$
where the notation $R_{B} = R_{B, B}$ was defined in \ref{eq:def_R_A_B}.  Let us now recall here that by the comment after equation \ref{eq:rel_Q_lin_map}, the vectors ${\bf X}_{Q^j}\in \R^{n^2}$ are eigenvectors of $\mathcal Q$ with eigenvalue $\lambda_j$, and $\langle {\bf X}_{Q^j}, {\bf X}_{Q^k} \rangle = \delta_{j, k}$ for each $j, k =1, ..., m$.  We deduce that
$$
{\mathcal Q} = \sum_{j=1}^m \lambda_j {\bf X}_{Q^j}{\bf X}_{Q^j}^T.
$$

We complete the first step of this proof by showing that the $Q^j$ are symmetric.  To do this we appeal to the permutation $\sigma_2$ from (\ref{distinguished_permutations}), and its operator $T_2$.  Equation \ref{eq:trans_first_mat} gives us
$$
{\mathcal Q} = T_2({\mathcal Q}) = \sum_{j=1}^m \lambda_j {\bf X}_{{Q^j}^T}{\bf X}_{Q^j}^T.
$$
It is not hard from here to deduce that in fact $Q^j = (Q^j)^T$ for $j=1, ..., m$.  So far then we have
$$
{\mathcal Q} = \sum_{j=1}^m \lambda_j {\bf X}_{Q^j}{\bf X}_{Q^j}^T
$$
with $\langle Q^i, Q^j\rangle = \delta_{ij}$ and $(Q^j)^T = Q^j$.  Note that the fact that $Q_j^T = Q_j$, $j=1, ..., m$, tells us that $1\leq m \leq \frac{n(n+1)}{2}$.  This finishes the first step.

In the second step we show that the $Q^j$ commute with each other, for which we proceed as follows.  First we observe that
$$
\mathcal{Q}^2 = \sum_{j=1}^m \lambda_j^2 \, \bX_{Q^j}\bX_{Q^j}^T.
$$
Since $T_3({\mathcal Q^2}) = {\mathcal Q^2}$, equation \ref{eq:cols_to_blocks} gives us
$$
{\mathcal Q^2} = T_3({\mathcal Q^2}) = \sum_{j=1}^m \lambda_j^2 T_3({\bf X}_{Q^j}{\bf X}_{Q^j}^T) = \sum_{j=1}^m \lambda_j^2 \left ( \begin{array}{ccc} Q^j_{11}Q^j & \cdots & Q^j_{1n}Q^j \\ \vdots & \ddots & \vdots \\ Q^j_{n1}Q^j & \cdots & Q^j_{nn}Q^j \end{array}  \right ).
$$
A direct computation from this last equation then shows that
$$
{\mathcal Q}^4 = \sum_{j=1}^m \lambda_j^4 \, {\bf X}_{Q^j}{\bf X}_{Q^j}^T = \sum_{i, j=1}^m \lambda_i^2\lambda_j^2 \left ( \begin{array}{ccc} (Q^iQ^j)_{11}Q^iQ^j & \cdots & (Q^iQ^j)_{1n}Q^iQ^j \\ \vdots & \ddots & \vdots \\ (Q^iQ^j)_{n1}Q^iQ^j & \cdots & (Q^iQ^j)_{nn}Q^iQ^j \end{array}  \right ).
$$
Since ${\mathcal Q}^4 = T_{3}({\mathcal Q}^4)$, again by \ref{eq:cols_to_blocks} we conclude that
$$
\sum_{j=1}^m \lambda_j^4 \left ( \begin{array}{ccc} Q^j_{11}Q^j & \cdots & Q^j_{1n}Q^j \\ \vdots & \ddots & \vdots \\ Q^j_{n1}Q^j & \cdots & Q^j_{nn}Q^j \end{array}  \right ) = \sum_{i, j=1}^m \lambda_i^2\lambda_j^2  \left ( \begin{array}{ccc} (Q^iQ^j)_{11}Q^iQ^j & \cdots & (Q^iQ^j)_{1n}Q^iQ^j \\ \vdots & \ddots & \vdots \\ (Q^iQ^j)_{n1}Q^iQ^j & \cdots & (Q^iQ^j)_{nn}Q^iQ^j \end{array}  \right ).
$$
Denote by ${\mathcal Q}_L^4 = \Phi_0({\mathcal Q}^4)$, the linear map associated to $\mathcal{Q}^4$.  From the computations above we obtain directly that, for any $A\in \Mall$, it holds
$$
{\mathcal Q}_L^4(A) = \sum_{j=1}^m \lambda_j^4 \langle Q^j, A\rangle Q^j = \sum_{j, k=1}^m \lambda_j^2\lambda_k^2 \langle Q^jQ^k, A \rangle Q^jQ^k.
$$
Next recall that
$$
(A, B) = AB + BA \,\,\,\, \mbox{and} \,\,\,\,  [A, B] = AB-BA.
$$
Since
$$
2Q^jQ^k = (Q^j, Q^k) + [Q^j, Q^k],
$$
it is easy to check that
$$
{\mathcal Q}^4_L(A) =  \sum_{j=1}^m \lambda_j^4 \langle Q^j, A\rangle Q^j = \frac{1}{4} \sum_{j, k=1}^m \lambda_j^2\lambda_k^2 \langle (Q^j, Q^k), A \rangle (Q^j, Q^k)  + \frac{1}{4} \sum_{j, k=1}^m \lambda_j^2\lambda_k^2 \langle [Q^j, Q^k], A \rangle [Q^j, Q^k].
$$

We are finally in a position to conclude that the $Q^j$ commute with each other.  For this consider an anti-symmetric $A\in \Mall$ in the above equation.  The first expression for ${\mathcal Q}_L^4$ gives ${\mathcal Q}_L^4(A)=0$ because $(Q^j)^T=Q^j$.  Then, since $(Q^j, Q^k)$ is symmetric, we get
\begin{equation} \label{eq:4thcomm}
0 = \frac{1}{4} \sum_{j, k=1}^m \lambda_j^2\lambda_k^2 \langle [Q^j, Q^k], A \rangle [Q^j, Q^k]= \frac{1}{2} \sum_{1\leq j < k \leq m} \lambda_j^2\lambda_k^2 \langle [Q^j, Q^k], A \rangle [Q^j, Q^k]
\end{equation}
for every anti-symmetric $A\in \Mall$.  Since $[Q^j, Q^k]$ is anti-symmetric, and $\lambda_j^2 > 0$, we deduce
$$
[Q^j, Q^k] = 0
$$
for all $j, k = 1, ..., m$.  This is of course the statement that the $Q^j$ commute with each other, and ends the second step of the proof.

From here it is now easy to conclude the proof of the theorem.  Indeed, since the $Q^j$ commute with each other and are symmetric, they have a common orthonormal basis of eigenvectors.

If we denote this basis by ${\bf a}^1$, ..., ${\bf a}^n$, and define the associated projections
$$
P^j = {\bf a}^j({\bf a}^j)^T,
$$
then the $Q_j$ are all linear combination of the $P_j$.  Note this tells us that $1\leq m \leq n$.  Further, an appeal to equation \ref{eq:cols_to_blocks} shows that $\mathcal{Q}_L = \Phi_0(\mathcal{Q})$ satisfies
$$
{\mathcal Q}_L(A) = \sum_{j=1}^m \lambda_j \langle A, Q^j\rangle Q^j = \sum_{j=1}^n \lambda_j Q^jAQ^j.
$$
This shows that the $P^j$ are all eigenvectors of ${\mathcal Q}_L$, which is equivalent to saying that the $\bX_{P^j}$ are all eigenvectors of $\mathcal{Q}$.  Because of this we can write
\begin{equation}\label{eq:end_odeco_proof}
\mathcal{Q} = \sum_{j=1}^n \mu_j \bX_{P^j}\bX_{P^j}^T
\end{equation}
for real numbers $\mu_j$, some of which may be $0$.

To conclude the proof, we point out that the additional assumption that $\mathcal{Q}^2 = \mathcal{Q}$ tells us that $\mu_j=1$ for $j=1, ..., m$.  Furthermore, the trace conditions included in the definition given by equation \ref{eq:mfld_first_def} tell us that $m=n$.  This concludes the proof of Theorem \ref{thm:limitrelax}.
\end{proof}

{
A direct consequence of the proof of Theorem \ref{thm:limitrelax} is an alternative formulation of odeco varieties. Odeco stands for orthogonally decomposable, which can be observed in the definition below from the fact that the matrices $P^j$ that appear there are all orthogonal projection matrices of rank $1$ with mutually orthogonal images.  We recall the definition of these varieties, see \cite{Robeva}:
\begin{multline}\label{eq:def_Robeva}
\mathbb{M}^n_{\text{odeco}}  = \left\{\mathcal{Q}\in \Mallnsqr : \mathcal{Q}= \sum_{j=1}^n \lambda_j \bX_{P^j}\bX_{P^j}^T,\,\,\,  P^j\in \mathbb{M}^n, {\rm tr}(P^j) = 1, \right. \\
\left. \,\,\, (P^j)^2=(P^j)^T=P^j,  \,\,\, P^jP^k = \delta_{jk}, \,\,\, \lambda_j \in \mathbb{R} \right\}.
\end{multline}
We now can state the following


\begin{corollary}
\label{cor:ODECO}
Definition \eqref{eq:odecodef} is equivalent to definition \eqref{eq:def_Robeva}.
\end{corollary}
\begin{proof}
The proof of this corollary is exactly the same as the proof of Theorem \ref{thm:limitrelax}, up to, and including equation \ref{eq:end_odeco_proof}.  Note also that the tensor $\mathcal{Q}=0$ can be trivially written in the form given by equation \ref{eq:end_odeco_proof}.  This shows that the variety defined in \ref{eq:odecodef} is a subset of that defined in \ref{eq:def_Robeva}.  The other inclusion is trivial, and so we conclude that these two sets are equal.

\end{proof}

}

\section{Appendix B: Evolution equations for $3$-cross fields.}
\label{ap:2}
In this appendix we present the system of partial differential equations that governs the gradient flow evolution of $3$-cross fields in our simulations.   
Taking variational derivatives of the functional in \eqref{eq:efin}, we arrive at the following system of equations
\begin{multline*}
    q_{1t}-\mathrm{div\,}\left(4(\nabla q_1+\nabla q_4)+\frac{1}{2}\nabla q_8)\right)=-\frac{1}{\varepsilon^2}(-57+64 q_1^3-80 q_5^2-102 q_6^2+61 q_8+48 q_2^2 (-3+4 q_4+q_8)\\+24 q_1^2 (-7+8
q_4+q_8)+8 q_3^2 (-7+8 q_4+q_8)+8 q_3 q_7 (-7+8 q_4+q_8)+8 q_2 (q_6 (-30+42 q_4+13
q_8)\\+2 q_7 (8 q_5+5 q_9))+2 q_1 (79+64 q_2^2+120 q_4^2+96 q_2 q_6+40 q_6^2\\+32 (q_3^2+q_3
q_7+q_7^2)+48 q_4 (-4+q_8)+9 (-4+q_8) q_8+8 (8 q_5^2+4 q_5 q_9+q_9^2))\\+2
(56 q_4^3-19 q_7^2-8 q_5^2 q_8+24 q_6^2 q_8-8 q_7^2 q_8-18 q_8^2+4 q_8^3\\+3 q_4^2
(-46+17 q_8)+32 q_6 q_7 q_9+(-11+4 q_8) q_9^2+4 q_5 (6 q_6 q_7+(-8+q_8) q_9)\\+q_4
(109+32 q_5^2+72 q_6^2+32 q_7^2+3 q_8 (-27+8 q_8)+40 q_5 q_9+16 q_9^2))),
\end{multline*}
\begin{multline*}
    q_{2t}-2\,\mathrm{div}\left(4 \nabla q_2+3 \nabla q_6\right)=-\frac{4}{\varepsilon^2}(32 q_2^3+72 q_2^2 q_6+24 q_3^2 q_6+8 q_1^2 (4 q_2+3 q_6)+2 q_3 (6 q_6
q_7\\+8 q_5 (-1+2 q_4+q_8)+q_9-2 q_4 q_9)+2 q_1 (12 q_2 (-3+4 q_4+q_8)\\+q_6
(-30+42 q_4+13 q_8)+2 q_7 (8 q_5+5 q_9))+q_2 (49+32 q_3^2+108 q_4^2\\+72 q_6^2+36 q_7^2+72
q_4 (-2+q_8)+q_8 (-51+20 q_8)+4 (8 q_5^2+4 q_5 q_9+5 q_9^2))\\+3 (8 q_6^3+4
q_5 q_7 (-3+4 q_4+q_8)+6 (-1+2 q_4) q_7 q_9+q_6 (34 q_4^2\\+q_4 (-47+28 q_8)+4
(4+2 q_5^2+2 q_7^2+q_8 (-5+2 q_8)+q_5 q_9+2 q_9^2)))),
\end{multline*}
\begin{multline*}
    q_{3t}-2\,\mathrm{div}\left(2 \nabla q_3+\nabla q_7\right)=-\frac{4}{\varepsilon^2}(16 q_3^3-18 q_5 q_6+36 q_4 q_5 q_6+2 (6+q_1 (-7+4 q_1)) q_7+24 q_3^2
q_7\\-29 q_4 q_7+16 q_1 q_4 q_7+14 q_4^2 q_7+32 q_5^2 q_7+8 q_6^2 q_7+8 q_7^3+20
q_5 q_6 q_8-16 q_7 q_8+2 q_1 q_7 q_8+20 q_4 q_7 q_8+8 q_7 q_8^2\\+2 q_6
q_9-4 q_4 q_6 q_9+28 q_5 q_7 q_9+8 q_7 q_9^2+2 q_2 (6 q_6 q_7\\+8 q_5
(-1+2 q_4+q_8)+q_9-2 q_4 q_9)+q_3 (15+32 q_2^2-40 q_4\\+48 q_2 q_6+20 q_6^2+24
q_7^2-11 q_8+4 (4 q_1^2+7 q_4^2+4 q_4 q_8+q_8^2\\+q_1 (-7+8 q_4+q_8))+4 (8
q_5^2+4 q_5 q_9+q_9^2))),
\end{multline*}
\begin{multline*}
    q_{4t}-\mathrm{div\,}\left(4 \nabla q_1+11 \nabla q_4+4 \nabla q_8\right)=-\frac{2}{\varepsilon^2}(-57+32 q_1^3-40 q_3^2+221 q_4+56 q_3^2 q_4-285 q_4^2\\+130 q_4^3-56 q_5^2+88 q_4
q_5^2+72 q_3 q_5 q_6-144 q_6^2+216 q_4 q_6^2-58 q_3 q_7+56 q_3 q_4 q_7\\+96
q_5 q_6 q_7-56 q_7^2+88 q_4 q_7^2+109 q_8+16 q_3^2 q_8-276 q_4 q_8\\+168 q_4^2
q_8+32 q_5^2 q_8+96 q_6^2 q_8+40 q_3 q_7 q_8+32 q_7^2 q_8-96 q_8^2\\+120 q_4
q_8^2+32 q_8^3+72 q_2^2 (-2+3 q_4+q_8)+24 q_1^2 (-4+5 q_4+q_8)\\+6 q_2 q_6 (-47+68 q_4+28
q_8)-8 q_6 (q_3-8 q_7) q_9+2 q_5 (-29+28 q_4+16 q_8) q_9\\+8 (-5+7 q_4+4 q_8) q_9^2+8
q_2 (8 q_3 q_5+12 q_5 q_7-q_3 q_9+9 q_7 q_9)\\+q_1 (109+168 q_4^2+24 (q_2+q_6)
(4 q_2+3 q_6)+32 (q_3^2+q_3 q_7+q_7^2)\\+3 q_8 (-27+8 q_8)+6 q_4 (-46+17 q_8)+8
(4 q_5^2+5 q_5 q_9+2 q_9^2))),
\end{multline*}
\begin{multline*}
    q_{5t}-2\,\mathrm{div}\left(4 \nabla q_5+\nabla q_9\right)=-\frac{4}{\varepsilon^2}(32 q_5^3+6 q_6 (q_3 (-3+6 q_4)+2 (-3+q_1+4 q_4) q_7)+4 q_6 (5 q_3+8 q_7)
q_8\\+24 q_5^2 q_9+(12+8 q_1^2+8 q_3^2-29 q_4+14 q_4^2+28 q_3 q_7+32 q_7^2+2 (-7+8
q_4) q_8\\+8 q_8^2+2 q_1 (-8+10 q_4+q_8)) q_9+8 q_9^3+8 q_2^2 (4 q_5+q_9)\\+q_5
(21+32 q_1^2+32 q_3^2+44 q_4^2+36 q_6^2+64 q_3 q_7+68 q_7^2\\+8 q_1 (-5+4 q_4-q_8)-19
q_8+16 q_8^2+8 q_4 (-7+4 q_8)+24 q_9^2)\\+4 q_2 (8 q_1 q_7+4 q_3 (-1+2 q_4+q_8)+3
q_7 (-3+4 q_4+q_8)+3 q_6 (4 q_5+q_9))),
\end{multline*}
\begin{multline*}
    q_{6t}-2\,\mathrm{div}\left(3 \nabla q_2+4 \nabla q_6\right)=-\frac{4}{\varepsilon^2}(24 q_2^3+(49+q_1 (-51+20 q_1)) q_6+72 q_2^2 q_6+20 q_3^2 q_6\\+2 q_3 (8 q_6
q_7+q_5 (-9+18 q_4+10 q_8)+q_9-2 q_4 q_9)+4 (27 q_4^2 q_6\\+9 q_5^2 q_6+8 q_6
(q_6^2+q_7^2)+6 (-3+q_1) q_6 q_8+8 q_6 q_8^2\\+q_5 q_7 (-9+3 q_1+8 q_8)+4
(-1+q_1) q_7 q_9+8 q_6 q_9^2+2 q_4 (6 q_5 q_7+3 q_6 (-6+3 q_1+4 q_8)\\+4 q_7
q_9))+q_2 (24 q_1^2+q_1 (-60+84 q_4+26 q_8)+3 (8 q_3^2+q_4 (-47+34 q_4)\\+4 q_3
q_7+28 q_4 q_8+8 q_8^2+4 (4+2 q_5^2+6 q_6^2+2 q_7^2-5 q_8+q_5 q_9+2 q_9^2)))),
\end{multline*}
\begin{multline*}
    q_{7t}-2\,\mathrm{div}\left(\nabla q_3+4 \nabla q_7\right)=-\frac{4}{\varepsilon^2}(8 q_3^3-36 q_2 q_5+21 q_7+16 q_1^2 q_7+24 q_3^2 q_7+4 (9 q_2^2 q_7\\+17
q_5^2 q_7+q_5 q_6 (-9+12 q_4+8 q_8)+3 q_2 (4 q_4 q_5+4 q_6 q_7+q_5 q_8)\\+q_7
(q_4 (-14+11 q_4)+8 (q_6^2+q_7^2)-10 q_8+8 q_4 q_8+8 q_8^2))\\+2 ((-1+2
q_4) (9 q_2+8 q_6)+32 q_5 q_7) q_9+32 q_7 q_9^2+q_1 (q_7 (-19+32 q_4-8 q_8)\\+4
(8 q_2 q_5+3 q_5 q_6+5 q_2 q_9+4 q_6 q_9))+q_3 (8 q_1^2+14 q_4^2\\+2 q_1
(-7+8 q_4+q_8)+q_4 (-29+20 q_8)+4 (3+8 q_5^2\\+q_6 (3 q_2+2 q_6)+6 q_7^2+2 (-2+q_8)
q_8+7 q_5 q_9+2 q_9^2))),
\end{multline*}
\begin{multline*}
    q_{8t}-\mathrm{div\,}\left(\frac{1}{2}\nabla q_1+4 (\nabla q_4+\nabla q_8)\right)=-\frac{1}{\varepsilon^2}(-57+8 q_1^3+218 q_4-38 q_5^2+158 q_8+6 q_1^2 (-6+8 q_4+3 q_8)\\+2 q_2^2 (-51+72
q_4+40 q_8)+16 q_2 (4 q_3 q_5+3 q_5 q_7+3 q_6 (-5+7 q_4+4 q_8))\\+q_1 (61+102
q_4^2+8 (3 q_2+2 q_6) (2 q_2+3 q_6)+8 (q_3-q_7) (q_3+2 q_7)\\+24 (-3+q_8) q_8+6 q_4
(-27+16 q_8)+8 (-2 q_5^2+q_5 q_9+q_9^2))+2 (56 q_4^3\\+q_3^2 (-11+16 q_4+8 q_8)+6
q_4^2 (-23+20 q_8)+8 q_3 (5 q_5 q_6+q_7 (-4+5 q_4+4 q_8))\\+32 q_4 (q_5^2+3 q_6^2+q_7^2+3
(-2+q_8) q_8+q_5 q_9+q_9^2)+4 (16 q_5 q_6 q_7+2 q_6^2 (-9+8 q_8)\\+2 q_7^2
(-5+8 q_8)+q_8 (8 q_5^2+q_8 (-21+8 q_8))+q_5 (-7+8 q_8) q_9+(-7+8 q_8) q_9^2))),
\end{multline*}
\begin{multline*}
    q_{9t}-2\,\mathrm{div}\left(\nabla q_5+2 \nabla q_9\right)=-\frac{4}{\varepsilon^2}(8 q_5^3+2 q_6 (q_3-2 q_3 q_4+8 (-1+q_1+2 q_4) q_7)+15 q_9+24 q_5^2
q_9\\+(4 q_1^2+q_1 (-11+16 q_4+4 q_8)+4 (q_3^2+q_4 (-10+7 q_4)+4 q_3 q_7+8 (q_6^2+q_7^2)-7
q_8\\+8 q_4 q_8+4 q_8^2)) q_9+16 q_9^3+4 q_2^2 (2 q_5+5 q_9)+q_5 (12+8
q_1^2\\+8 q_3^2-29 q_4+14 q_4^2+28 q_3 q_7+32 q_7^2-14 q_8+16 q_4 q_8+8 q_8^2+2 q_1
(-8+10 q_4+q_8)\\+24 q_9^2)+2 q_2 (q_3-2 q_3 q_4+(-9+10 q_1+18 q_4) q_7+6 q_6
(q_5+4 q_9))).
\end{multline*}
This system is solved subject to the boundary constraints \eqref{eq:bcsys} and the natural (Robin) boundary conditions arising in the variational problem for the functional \eqref{eq:efin}.
\bibliographystyle{abbrv}

\bibliography{GLFbib}

\end{document}